\documentclass[11pt]{amsart}

\usepackage{amsfonts,epsfig}
\usepackage{latexsym}
\usepackage{amssymb}
\usepackage{amsmath}
\usepackage{amsthm}
\usepackage{graphics}
\usepackage[all]{xy}
\usepackage[T2A]{fontenc}
\usepackage{multirow}
\usepackage{array, booktabs, ctable}
\usepackage{hyperref}
\usepackage{color}
\newtheorem{defn}{Definition}[section]

\newtheorem{lemma}[defn]{Lemma}

\newtheorem{thm}[defn]{Theorem}
\newtheorem{theorem}[defn]{Theorem}
\newtheorem{cor}[defn]{Corollary}
\newtheorem{prop}[defn]{Proposition}
\newtheorem{proposition}[defn]{Proposition}

\theoremstyle{definition}

\newcommand{\Q}{\mathbb Q}
\newcommand{\Z}{\mathbb Z}

\newcommand{\image}{\operatorname{im}}

\newcommand{\SL}{\operatorname{SL}}

\newcommand{\Gal}{\operatorname{Gal}}

\newcommand{\GL}{\operatorname{GL}}

\begin{document}



\pagenumbering{arabic}
\bibliographystyle{plain}
\title[Torsion Quartic Number Fields]{Torsion of Rational Elliptic Curves over Quartic Galois Number Fields}

\author{Michael Chou}

\address{Dept. of Mathematics, Univ. of Connecticut, Storrs, CT 06269, USA}
\email{michael.chou@uconn.edu} 

%\keywords{modular curves}

%\subjclass{Primary: 11G05, Secondary: 14H52.}

% Authors and running title to go on top of each page
%\pagestyle{myheadings} \markboth{\'Alvaro Lozano-Robledo}{Rank
%over large fields.}

\begin{abstract} Let $E$ be an elliptic curve defined over $\Q$, and let $K$ be a number field of degree four that is Galois over $\mathbb{Q}$.  The goal of this article is to classify the different isomorphism types of $E(K)_{\text{tors}}$.
\end{abstract}

\maketitle

%\part{Use this type of header for very long papers only}

%%%%%%%%%%%%%%%%%%%%%%%%%%%%%%%%%%%%%%%%%%%%%%%%%%%%%%%%%%%%%%%%%%%%%%%%%%%%%%%%
%%%%%%%%%%%%%%%%%%%%%%%%%%%%%%%%%%%%%%%%%%%%%%%%%%%%%%%%%%%%%%%%%%%%%%%%%%%%%%%%
%%%%%%%%%%%%%%%%%%%%%%%%%%%%%%%%%%%%%%%%%%%%%%%%%%%%%%%%%%%%%%%%%%%%%%%%%%%%%%%%
\section{Introduction}

Let $E$ be an elliptic curve defined over $\mathbb{Q}$.  Given a number field $K$, we may consider $E$ as an elliptic curve defined over $K$ and examine the structure of the points of $E$ with coordinates in $K$, denoted $E(K)$.  We have the following fundamental theorem describing the structure of $E(K)$:

\begin{theorem}[Mordell-Weil]
Let $E$ be an elliptic curve over a number field $K$. The group of $K$-rational points, $E(K)$, is a finitely generated abelian group.
\end{theorem}

By the fundamental theorem of finitely generated abelian groups it follows that, for any elliptic curve $E$ over $K$, there exists an integer $r_{K}> 0$ depending on $K$ such that
$$E(K) \cong E(K)_{\text{tors}} \oplus \mathbb{Z}^{r_{K}}$$
where $E(K)_{\text{tors}}$ is a finite group.  We call $r_{K}$ the rank of $E$ over $K$, and we call $E(K)_{\text{tors}}$ the torsion subgroup of the $E$ over $K$.  A natural question is which groups can arise as torsion subgroups of elliptic curves over certain number fields.

  In this paper we obtain a classification of the torsion subgroup of elliptic curves with rational coefficients over number fields $K$ that are quartic Galois extensions of $\mathbb{Q}$.  We separate the classification based on the isomorphism type of $\Gal(K/\mathbb{Q})$.  If $\Gal(K/\mathbb{Q}) \cong \mathbb{Z}/4\mathbb{Z}$ we call $K$ a cyclic quartic extension of $\mathbb{Q}$, and if $\Gal(K/\mathbb{Q}) \cong \mathbb{Z}/2\mathbb{Z} \oplus \mathbb{Z}/2\mathbb{Z}$ we call $K$ a biquadratic extension of $\mathbb{Q}$.
  
    The main results of this article are as follows:
    
\begin{theorem}\label{quarticgalois}
Let $E/\mathbb{Q}$ be an elliptic curve, and let $K$ be a quartic Galois extension of $\mathbb{Q}$.  Then $E(K)_{tors}$ is isomorphic to one of the following groups:
$$
\begin{array}{ll}
\mathbb{Z} / N_1 \mathbb{Z}, & N_1 = 1, \ldots ,16 , N_1 \neq 11, 14,\\
\mathbb{Z} / 2\mathbb{Z} \oplus \mathbb{Z}/2N_2 \mathbb{Z}, & N_2 = 1, \ldots , 6, 8,\\
\mathbb{Z} /3 \mathbb{Z} \oplus \mathbb{Z}/3N_3 \mathbb{Z}, & N_3 = 1,2, \\
\mathbb{Z} / 4\mathbb{Z} \oplus \mathbb{Z}/4N_4 \mathbb{Z}, & N_4 = 1, 2, \\
\mathbb{Z} / 5\mathbb{Z} \oplus \mathbb{Z}/ 5\mathbb{Z}, & \\
\mathbb{Z} / 6\mathbb{Z} \oplus \mathbb{Z}/ 6\mathbb{Z}. & \\
\end{array}
$$
Each of these groups, except for $\mathbb{Z}/15\mathbb{Z}$, appears as the torsion structure over some quartic Galois field for infinitely many (non-isomorphic) elliptic curves defined over $\mathbb{Q}$. 
\end{theorem}

The proof of this theorem is broken up based on the structure of $\Gal(K/\mathbb{Q})$ and so, in fact, we have the following more specialized theorems:

\begin{theorem}\label{cyclicquartic}
Let $E/\mathbb{Q}$ be an elliptic curve, and let $K$ be a quartic Galois extension with $\Gal(K/\mathbb{Q}) \cong \mathbb{Z}/4\mathbb{Z}$.  Then $E(K)_{tors}$ is isomorphic to one of the following groups:
$$
\begin{array}{ll}
\mathbb{Z} / N_1 \mathbb{Z}, & N_1 = 1, \ldots , 10, 12, 13, 15 ,16,\\
\mathbb{Z} / 2\mathbb{Z} \oplus \mathbb{Z}/2N_2 \mathbb{Z}, & N_2 = 1, \ldots , 6, 8,\\
\mathbb{Z} / 5\mathbb{Z} \oplus \mathbb{Z}/ 5\mathbb{Z}. & \\
\end{array}
$$
\end{theorem}

\begin{theorem}\label{biquad}
Let $E/\mathbb{Q}$ be an elliptic curve, and let $K$ be a quartic Galois extension with $\Gal(K/\mathbb{Q}) \cong \mathbb{Z}/2\mathbb{Z} \oplus \mathbb{Z}/2\mathbb{Z}$.  Then $E(K)_{tors}$ is isomorphic to one of the following groups:
$$
\begin{array}{ll}
\mathbb{Z} / N_1 \mathbb{Z}, & N_1 = 1, \ldots ,10 , 12, 15, 16,\\
\mathbb{Z} / 2\mathbb{Z} \oplus \mathbb{Z}/2N_2 \mathbb{Z}, & N_2 = 1, \ldots , 6, 8,\\
\mathbb{Z} /3 \mathbb{Z} \oplus \mathbb{Z}/3N_3 \mathbb{Z}, & N_3 = 1,2, \\
\mathbb{Z} / 4\mathbb{Z} \oplus \mathbb{Z}/4N_4 \mathbb{Z}, & N_4 = 1, 2, \\
\mathbb{Z} / 6\mathbb{Z} \oplus \mathbb{Z}/ 6\mathbb{Z}. & \\
\end{array}
$$
\end{theorem}
  
The proof of Theorem \ref{biquad} is simply a collection of known results, and so we give a simple explanation of the proof in Section \ref{biquadproof}.  The proof of Theorem \ref{cyclicquartic} is more involved.  Section \ref{auxresults} will show the tools involved in the proof, and Section \ref{cyclicquadproof} will contain the actual proof.  Finally, Section \ref{examples} will contain examples and references of infinite families for each torsion subgroup appearing in Theorem \ref{quarticgalois} to finish the proof.

This work is inspired by the many classifications of torsion subgroups of elliptic curves over number fields.  To begin. the classification of torsion structures of elliptic curves over $\mathbb{Q}$ was done by Mazur \cite{mazur1} who proved the following theorem on the torsion of an elliptic curve over $\mathbb{Q}$.

\begin{theorem}[Mazur \cite{mazur1}]
Let $E$ be an elliptic curve defined over $\mathbb{Q}$.  Then $E(\mathbb{Q})_{tors}$ is isomorphic to one of the following 15 groups:
\begin{align*}
&\mathbb{Z}/N_1\mathbb{Z}, &1 \leq N_1 \leq 12, N_1 \neq 11,\\
&\mathbb{Z}/2\mathbb{Z} \oplus \mathbb{Z}/ 2N_2\mathbb{Z}, &1 \leq N_2 \leq 4.
\end{align*}
\end{theorem}

Each of these groups actually appear as the torsion subgroup of infinitely many (isomorphism classes of) elliptic curves over the rationals.

The classification of torsion subgroups of elliptic curves defined over quadratic number fields was answered by Kamienny, Kenku, and Momose:

\begin{theorem}[Kamienny \cite{kamienny}, Kenku, Momose \cite{kenkumomose}]\label{kkmquad}
Let $K$ be a quadratic number field, and let $E$ be an elliptic curve defined over $K$.  Then $E(K)_{\text{tors}}$ is isomorphic to one of the following groups:
$$
\begin{array}{ll}
\mathbb{Z}/N_1\mathbb{Z}, &1 \leq N_1 \leq 18, N_1 \neq 17,\\
\mathbb{Z}/2\mathbb{Z} \oplus \mathbb{Z}/ 2N_2\mathbb{Z}, &1 \leq N_2 \leq 6,\\
\mathbb{Z}/3\mathbb{Z} \oplus \mathbb{Z}/ 3N_3\mathbb{Z}, &N_3=1,2,\\
\mathbb{Z}/4\mathbb{Z} \oplus \mathbb{Z}/ 4\mathbb{Z}.
\end{array}
$$
\end{theorem}

Again, each of these groups actually appear as the torsion subgroup of infinitely many (isomorphism classes of) elliptic curves over quadratic number fields.

However, already in degree 3 number fields, we have yet to find a complete classification of the torsion subgroups of elliptic curves defined over cubic number fields.

For number fields of degree 3, the torsion subgroups of elliptic curves defined over cubic number fields appearing infinitely often for non-isomorphic elliptic curves have been classified by Joen, Kim, and Schweizer.

\begin{theorem}[Jeon, Kim, Schweizer \cite{jeonkimschweizer}]
Let $K$ be a cubic number field, and let $E$ be an elliptic curve defined over $K$.  The groups appearing as $E(K)_{\text{tors}}$ infinitely often for non-isomorphic elliptic curves are precisely:
$$
\begin{array}{ll}
\mathbb{Z}/N_1\mathbb{Z}, &1 \leq N_1 \leq 20, N_1 \neq 17,19,\\
\mathbb{Z}/2\mathbb{Z} \oplus \mathbb{Z}/ 2N_2\mathbb{Z}, &1 \leq N_2 \leq 7.\\
\end{array}
$$
\end{theorem}

From the case over $\mathbb{Q}$ and the case over quadratic number fields, we might expect this to be a full list of all torsion subgroups of elliptic curves over cubic number fields.  However, Najman showed that the curve $$y^2 + xy + y = x^3 - x^2 - 5x - 5$$ has torsion subgroup isomorphic to $\mathbb{Z}/21\mathbb{Z}$ over the cubic field $\mathbb{Q}(\zeta_{9})^{+}$, so the above list must be incomplete.  It remains to see whether there are any other torsion subgroups appearing in only finitely many elliptic curves over cubic number fields.\\

Jeon, Kim, and Park have also determined which groups appear as $E(K)_{\text{tors}}$ infinitely often for quartic number fields $K$ but, similar to the cubic case, there may be other torsion subgroups appearing only finitely many times.\\

\begin{theorem}[Jeon, Kim, Park, \cite{jeonkimpark} Theorem 3.6]
If $K$ varies over all quartic number fields and $E$ varies over all elliptic curves defined over $K$, the group structures that appear infinitely often as $E(K)_{tors}$ are exactly the following:

$$
\begin{array}{ll}
\mathbb{Z} / N_1 \mathbb{Z}, & N_1 = 1, \ldots ,18, 20, 21, 22, 24 \\
\mathbb{Z} / 2\mathbb{Z} \oplus \mathbb{Z}/2N_2 \mathbb{Z}, & N_2 = 1, \ldots , 9 \\
\mathbb{Z} / 3\mathbb{Z} \oplus \mathbb{Z}/3N_3 \mathbb{Z}, & N_3 = 1, 2, 3 \\
\mathbb{Z} / 4\mathbb{Z} \oplus \mathbb{Z}/4N_4 \mathbb{Z}, & N_4 = 1, 2 \\
\mathbb{Z} / 5\mathbb{Z} \oplus \mathbb{Z}/ 5\mathbb{Z}, & \\
\mathbb{Z} / 6\mathbb{Z} \oplus \mathbb{Z}/ 6\mathbb{Z}.
\end{array}
$$

In fact, all of these torsion structures already occur infinitely often if $K$ varies over all quadratic extensions of all quadratic number fields, that is, all biquadratic number fields.
\end{theorem} 

In all of the above cases we have allowed our elliptic curves to be defined over the number fields themselves.  However, it is often in one's interest to only consider elliptic curves defined over $\mathbb{Q}$ and see what torsion subgroups arise if we base extend these curves to higher degree number fields.  In a paper by Najman, the classification of torsion structures of rational elliptic curves over quadratic number fields is given:

\begin{theorem}[Najman, \cite{najman}, Theorem 2]\label{quadcase}
Let $E$ be an elliptic curve defined over $\mathbb{Q}$, and let $F$ be a quadratic number field.  Then$E(F)_\text{tors}$ is isomorphic to one of the following groups
$$
\begin{array}{ll}
\mathbb{Z}/N_1\mathbb{Z}, &N_1 =1, \ldots, 10, 12, 15, 16 \\
\mathbb{Z}/2\mathbb{Z} \oplus \mathbb{Z}/2N_2\mathbb{Z}, &1 \leq N_2 \leq 6 \\
\mathbb{Z}/3\mathbb{Z} \oplus \mathbb{Z}/3N_3\mathbb{Z}, &N_3 = 1,2\\
\mathbb{Z}/4\mathbb{Z} \oplus \mathbb{Z}/4\mathbb{Z}.
\end{array}
$$
Each of these groups, except for $\mathbb{Z}/15\mathbb{Z}$, appears as a torsion structure over a quadratic field for infinitely many rational elliptic curves $E$.
\end{theorem}

The main result of \cite{najman} is the classification of torsion structures of rational elliptic curves over cubic number fields:

\begin{theorem}[Najman, \cite{najman}, Theorem 1]\label{cubecase}
Let $E$ be an elliptic curve defined over $\Q$, and let $K / \mathbb{Q}$ be a cubic extension.  Then, $E(K)_{\text{tors}}$ is one of the following groups:
$$
\begin{array}{ll}
\mathbb{Z}/N_1\mathbb{Z}, &N_1 = 1, \ldots, 10,12,13,14,18,21 \\
\mathbb{Z}/2\mathbb{Z} \oplus \mathbb{Z}/2N_2\mathbb{Z}, &N_2 = 1,2,3,4,7. 
\end{array}
$$
\end{theorem}

More generally, in \cite{lozanorobledo1} there is a conjectural formula for each $d>0$ for the set of primes dividing the size of the $E(K)_{\text{tors}}$, where $K$ is a  number field of degree $d$, and $E$ an elliptic curve defined over $\mathbb{Q}$.

Ideally a similar classification as in Theorem \ref{cubecase} may be achieved for all quartic number fields.  This article puts forth the first steps by classifying torsion subgroups over a subset of quartic number fields, namely number fields that are Galois over $\mathbb{Q}$.

\textbf{Acknowledgments} This work was inspired by the paper \cite{najman} by Filip Najman, and the author would like to thank him for his ideas and help throughout the project.  The author would also like to thank \'{A}lvaro Lozano-Robledo for his invaluable guidance and input.  The author sends many thanks to the anonymous referee for their helpful comments and the quick refereeing process.  Lastly, the author would also like to thank Yasutsugu Fujita, Enrique Gonz\'{a}lez-Jim\'{e}nez, Andrew Sutherland, and David Zywina, for offering their advice and pointing the author to helpful references.

%Further, in a paper by Jeon, Kim, Lee, \cite{jeonkimlee} they give constructions of infinite families of curves with each prescribed torsion subgroup.  In fact, many of these constructions give elliptic curves defined over $\mathbb{Q}$, thus helping to narrow down the problem of what torsion subgroups appear as $E(K)_{tors}$ for elliptic curves defined over the rationals, $K$ a quartic number field.

%For the following torsion subgroups, however, it remains to be seen whether they occur as torsion subgroups in quartic number fields of some elliptic curve defined over $\mathbb{Q}$:

%$$
%\begin{array}{ll}
%\mathbb{Z} / N_1 \mathbb{Z}, & N_1 = 11,13, \ldots, 18, 20, 21, 22, 24\\
%\mathbb{Z} / 2 \mathbb{Z} \oplus \mathbb{Z} / 2N_2 \mathbb{Z}, & N_2 = 7, 8, 9\\
%\mathbb{Z} / 3 \mathbb{Z} \oplus \mathbb{Z} / 9 \mathbb{Z}
%\end{array}
%$$

%Moreover, this only classifies group structures appearing infinitely often.  We must also answer the question of which group structures could appear in some exceptional cases finitely many times.

%\begin{ack}
%I would like to thank Kevin Buzzard, Fred Diamond, Benedict Gross, and Felipe Voloch for their help, and suggestions. I would like to give special thanks to Robert Pollack for pointing out that one can use Serre's conjecture to prove Theorem \ref{thm-serreramif}.
%\end{ack}

\section{Torsion over biquadratic number fields}\label{biquadproof}

Fujita classified the torsion of elliptic curves defined over $\mathbb{Q}$ extended to the maximal elementary 2-abelian extension $L = \mathbb{Q}(\{\sqrt{m} : m \in \mathbb{Z} \})$.

\begin{theorem}[Fujita \cite{fujita}, Theorem 2]\label{fujita}
Let $E$ be an elliptic curve defined over $\mathbb{Q}$.  Let $L := \mathbb{Q}(\{\sqrt{m} : m \in \mathbb{Z} \})$.  Then the torsion subgroup $E(L)_{tors}$ is isomorphic to one of the following 20 groups:
$$
\begin{array}{ll}
\mathbb{Z}/2\mathbb{Z} \oplus \mathbb{Z}/2N\mathbb{Z}, & N = 1, 2, 3, 4, 5, 6, 8, \\
\mathbb{Z}/4\mathbb{Z} \oplus \mathbb{Z}/4N\mathbb{Z}, & N = 1,2,3,4, \\
\mathbb{Z}/2N\mathbb{Z} \oplus \mathbb{Z}/2N\mathbb{Z}, & N = 3,4\\
\end{array}
$$
or $\{0\},\  \mathbb{Z}/3\mathbb{Z},\  \mathbb{Z}/3\mathbb{Z} \oplus \mathbb{Z}/3\mathbb{Z},\  \mathbb{Z}/5\mathbb{Z},\  \mathbb{Z}/7\mathbb{Z},\  \mathbb{Z}/9\mathbb{Z},\  \mathbb{Z}/15\mathbb{Z}$.  Moreover, there exists an elliptic curve $E$ over $\mathbb{Q}$ which realizes each group listed above as $E(L)_{tors}$.
\end{theorem}

Certainly, if $E$ is an elliptic curve over $\mathbb{Q}$ and $K = \mathbb{Q}(\sqrt{m},\sqrt{n})$ for some square-free $m,n \in \mathbb{Z}$, then $E(K)_{tors} \leq E(L)_{tors}$, and so this gives the following list of possibilities for $E(K)_{tors}$:

$$
\begin{array}{ll}
\mathbb{Z}/N\mathbb{Z} & N = 1, \ldots , 10, 12, 15, 16.\\
\mathbb{Z}/2\mathbb{Z} \oplus \mathbb{Z}/2N\mathbb{Z}, & N = 1, 2, 3, 4, 5, 6, 8, \\
\mathbb{Z}/4\mathbb{Z} \oplus \mathbb{Z}/4N\mathbb{Z}, & N = 1,2,3,4, \\
\mathbb{Z}/2N\mathbb{Z} \oplus \mathbb{Z}/2N\mathbb{Z}, & N = 3,4,\\
\mathbb{Z}/3\mathbb{Z} \oplus \mathbb{Z}/3\mathbb{Z}, \\
\mathbb{Z}/3\mathbb{Z} \oplus \mathbb{Z}/6\mathbb{Z}.
\end{array}
$$

First we wish to show that any group in Najman's classification in Theorem \ref{quadcase} must also appear over a biquadratic extension of $\mathbb{Q}$ by adjoining a square root that will not add any torsion points.  This is done with the following lemma:

\begin{lemma}\label{biquadaddnotorsion}
Let $E$ be an elliptic curve defined over $\mathbb{Q}$.  Given a quadratic field $F_1 / \mathbb{Q}$, there is a quadratic field $F_2 / \mathbb{Q}$ such that $E(F_1 F_2)_{\text{tors}} \cong E(F_1)_{\text{tors}}$.
\end{lemma}
\begin{proof}
Let $E$ be an elliptic curve defined over $\mathbb{Q}$, and let $F_1$ be a quadratic field.  Let $L$ denote the maximal elementary 2-abelian extension of $\mathbb{Q}$.  Then, $E(F_1)_{\text{tors}} \subseteq E(L)_{\text{tors}}$.  Let $M$ be the field of definition of all the points in $E(L)_{\text{tors}}$.  Since by Theorem \ref{fujita} there are only finitely many points in $E(L)_{\text{tors}}$, it follows that $M$ is a finite extension of $\mathbb{Q}$ with $E(M)_{\text{tors}} \cong E(L)_{\text{tors}}$.  Pick a $d \in \mathbb{Z}$ such that $\sqrt{d} \not\in M$.  Then
$$E(L)_{\text{tors}} \cong E(M)_{\text{tors}} \cong E(M(\sqrt{d}))_{\text{tors}},$$
that is, no torsion points are gained by adjoining $\sqrt{d}$.  Thus, setting $F_2 = \mathbb{Q}(\sqrt{d})$, we have that $E(F_1)_{\text{tors}} \cong E(F_1 F_2)_{\text{tors}}$.
\end{proof}

This leaves us with six cases to verify:

$$
\begin{array}{ll}
\mathbb{Z}/2\mathbb{Z} \oplus \mathbb{Z}/16\mathbb{Z} &\\
\mathbb{Z}/4\mathbb{Z} \oplus \mathbb{Z}/4N\mathbb{Z}, & N = 2,3,4, \\
\mathbb{Z}/6\mathbb{Z} \oplus \mathbb{Z}/6\mathbb{Z}, & \\
\mathbb{Z}/8\mathbb{Z} \oplus \mathbb{Z}/8\mathbb{Z}.
\end{array}
$$

In \cite{jeonkimlee} Theorem 4.10 and Theorem 4.11, Jeon, Kim, Lee construct an infinite family of curves defined over $\mathbb{Q}$ such that $E(K)_{tors} \cong \mathbb{Z}/4\mathbb{Z} \oplus \mathbb{Z}/8\mathbb{Z}$ and $E(K)_{tors} \cong \mathbb{Z}/6\mathbb{Z} \oplus \mathbb{Z}/6\mathbb{Z}$ for biquadratic fields $K$ (see section \ref{examples} for infinite families of such curves).  Further, in Fujita's paper \cite{fujita} he gives an example of a curve $E : y^2 = x(x^2-47x+16^3)$ such that $E(L)_{tors} \cong \mathbb{Z}/2\mathbb{Z} \oplus \mathbb{Z}/16\mathbb{Z}$, but in fact already over a biquadratic extension $K=\mathbb{Q}(\sqrt{-7},\sqrt{-15})$, we have $E(K)_{tors} \cong \mathbb{Z}/2\mathbb{Z} \oplus \mathbb{Z}/16\mathbb{Z}$.

Thus, we need only to verify three cases:

$$\mathbb{Z}/4\mathbb{Z} \oplus \mathbb{Z}/12\mathbb{Z} ,\;  \mathbb{Z}/4\mathbb{Z} \oplus \mathbb{Z}/16\mathbb{Z},\;  \mathbb{Z}/8\mathbb{Z} \oplus \mathbb{Z}/8\mathbb{Z}.$$

\begin{proposition}
Let $E$ be an elliptic curve defined over $\mathbb{Q}$, and let $K$ be a biquadratic number field.  Then, $E(K)_{\text{tors}}$ is not isomorphic to $\mathbb{Z}/8\mathbb{Z} \oplus \mathbb{Z}/8\mathbb{Z}$.
\end{proposition}
\begin{proof}
Let $L$ denote the maximal elementary 2-abelian extension of $\mathbb{Q}$.  Suppose $E(K)_{tors} \cong \mathbb{Z}/8\mathbb{Z} \oplus \mathbb{Z}/8\mathbb{Z}$.  From \cite{fujita}, Proposition 11, if $E(\mathbb{Q})_{\text{tors}}$ is cyclic, then we have $E(L)_{\text{tors}} \not\supseteq \mathbb{Z}/8\mathbb{Z} \oplus \mathbb{Z}/8\mathbb{Z}$.  Since $E(K)_{\text{tors}} \subseteq E(L)_{\text{tors}}$, this implies that if $E(\mathbb{Q})_{\text{tors}}$ is cyclic, then $E(K)_{\text{tors}} \not\supseteq \mathbb{Z}/8\mathbb{Z} \oplus \mathbb{Z}/8\mathbb{Z}$.  
Thus, we may assume that $E(\mathbb{Q})_{\text{tors}}$ is not cyclic.  From \cite{fujita} in his closing remarks shows that $\mathbb{Z}/8\mathbb{Z} \oplus \mathbb{Z}/8\mathbb{Z}$ occurs in number fields of degree 16 or greater if $E$ is an elliptic curve over $\mathbb{Q}$ with non-cyclic torsion over $\mathbb{Q}$.
\end{proof}

Finally Najman and Bruin show in \cite{bruinnajman} that $\mathbb{Z}/4\mathbb{Z} \oplus \mathbb{Z}/12\mathbb{Z}$ and  $\mathbb{Z}/4\mathbb{Z} \oplus \mathbb{Z}/16\mathbb{Z}$ do not appear as torsion subgroups for any elliptic curves over any quartic number fields $K$ (not just elliptic curves defined over $\mathbb{Q}$).

\begin{thm}[Bruin, Najman, \cite{bruinnajman}, Theorem 7]
The following groups do not occur as subgroups of elliptic curves over quartic fields:
\begin{align*}
&\mathbb{Z}/3\mathbb{Z} \oplus \mathbb{Z}/12\mathbb{Z}, \;\; \mathbb{Z}/3\mathbb{Z} \oplus \mathbb{Z}/18\mathbb{Z}, \;\; \mathbb{Z}/3\mathbb{Z} \oplus \mathbb{Z}/27\mathbb{Z},\\
&\mathbb{Z}/3\mathbb{Z} \oplus \mathbb{Z}/33\mathbb{Z}, \;\; \mathbb{Z}/3\mathbb{Z} \oplus \mathbb{Z}/39\mathbb{Z}, \;\; \mathbb{Z}/4\mathbb{Z} \oplus \mathbb{Z}/12\mathbb{Z}, \\
&\mathbb{Z}/4\mathbb{Z} \oplus \mathbb{Z}/16\mathbb{Z}, \;\; \mathbb{Z}/4\mathbb{Z} \oplus \mathbb{Z}/28\mathbb{Z}, \;\; \mathbb{Z}/4\mathbb{Z} \oplus \mathbb{Z}/44\mathbb{Z},\\
&\mathbb{Z}/4\mathbb{Z} \oplus \mathbb{Z}/52\mathbb{Z}, \;\; \mathbb{Z}/4\mathbb{Z} \oplus \mathbb{Z}/68\mathbb{Z}, \;\; \mathbb{Z}/8\mathbb{Z} \oplus \mathbb{Z}/8\mathbb{Z}.
\end{align*}
\end{thm}

\section{Auxillary Results}\label{auxresults}

In this section we show a series of results needed to in the proof of Theorem \ref{cyclicquartic}.  We begin by proving a lemma for determining when a quartic extension is cyclic.

\begin{lemma}\label{criteriaforcyclic}  Let $K$ be a quartic Galois number field. Then $K$ can be written in the form $K = \mathbb{Q}(\sqrt{m})(\sqrt{\alpha})$ for some square free $m \in \mathbb{Q}$ and some $\alpha \in \mathbb{Q}(\sqrt{m})$.  Writing $\alpha = a + b\sqrt{m}$ for some $a,b \in \mathbb{Q}$, we have the following:
$$
\Gal(K/\mathbb{Q}) \cong \mathbb{Z}/4\mathbb{Z} \Leftrightarrow \frac{a^2}{m} - b^2=  1  \mbox{ for some }  a \neq 0, b \neq 0
$$
\end{lemma}
\begin{proof}
Notice that if $b = 0$ then $K$ is a biquadratic extension of $\mathbb{Q}$ and so $\Gal(K / \mathbb{Q}) \cong V_4$.  If $a=0$, then $K = \mathbb{Q}(\sqrt[4]{m})$ which is not a Galois extension of $\mathbb{Q}$.
Suppose $\Gal(K/ \mathbb{Q}) \cong \mathbb{Z}/4\mathbb{Z}$, i.e. $K$ is cyclic Galois over $\mathbb{Q}$.  Then the field $K$ must contain both $\sqrt{a+b\sqrt{m}}$ and $\sqrt{a-b\sqrt{m}}$ since they are conjugates.  For ease of notation we will call $\sqrt{a+b\sqrt{m}} = \beta$ and $\sqrt{a-b\sqrt{m}} = \overline{\beta}$. The other two conjugates differ only by a negative sign so we may claim that $K$ is a Galois extension if and only if $\overline{\beta} \in K$.\\
So we have that $K$ is Galois if and only if $\mathbb{Q}(\beta)=\mathbb{Q}(\overline{\beta})$.  This holds if and only if
$$ \sqrt{a+b\sqrt{m}} = c^2 \sqrt{a-b\sqrt{m}} $$
for some $c \in \mathbb{Q}(\sqrt{m})$, i.e.
$$ \frac{\sqrt{a+b\sqrt{m}}}{\sqrt{a-b\sqrt{m}}} = \frac{(a+b\sqrt{m})^2}{a^2-b^2 m} = c^2 $$
so if and only if $a^2 - b^2 m$ is a square in $\mathbb{Q}(\sqrt{m})$.
Writing this out we have:
$$a^2 - b^2 m = (d + e\sqrt{m})^2 = d^2 + 2ed\sqrt{m} + e^2m$$
for some $d,e \in \mathbb{Q}$. Either $d=0$ or $e=0$.  If $e=0$, then we have $a^2 - b^2 m = d^2$.  Then $\beta \overline{\beta} = \sqrt{a^2 - b^2 m} = d$, and some $\overline{\beta} = \frac{d}{\beta}$.  However, then any $\sigma$ in the Galois group of $K$ has order at most 2, since either $\sigma$ sends $\beta$ to $-\beta$ in which case it is clear that $\sigma$ has order two, or $\sigma$ sends $\beta$ to $\overline{\beta}$ in which case
$$\sigma^{2}(\beta) = \sigma(\overline{\beta}) = \sigma\left(\frac{d}{\beta}\right) = \frac{d}{\sigma{\beta}} = \beta.$$
Thus $K$ is Galois cyclic if and only if $a^2 - b^2 m=  e^2 m \mbox{ for some } e \in \mathbb{Q}, a \neq 0, b \neq 0$.

Dividing through by $e^2$ and absorbing it into $a^2$ and $b^2$ we get a nicer criterion: $$\frac{a^2}{m} - b^2 = 1.$$

\end{proof}

We use this to prove the next lemma, which in turn limits the amount of $n$-torsion that appears in $E(K)$.

\begin{lemma}\label{quadtotallyreal}
Let $K$ be a cyclic quartic extension of $\mathbb{Q}$.  Let $F$ be the (unique) intermediate quadratic subfield of $K$.  Then, $F$ is totally real.
\end{lemma}
\begin{proof}
From above, if $F = \mathbb{Q}(\sqrt{m})$, then a quadratic extension of $F$, say $F(\sqrt{a+b\sqrt{m}})$ for some $a,b \in \mathbb{Q}$, is quartic cyclic if and only if $\frac{a^2}{m} - b^2 = 1$.  Suppose $m < 0$, then there are no solutions over the reals to $\frac{a^2}{m}-b^2 = 1$, and so there is no quadratic extension of $F$ that is a cyclic quartic number field.  Therefore, if $F$ is contained in a cyclic quartic extension, $F = \mathbb{Q}(\sqrt{m})$ for some $m >0$, i.e. $F$ is totally real.
\end{proof}

\begin{lemma}\label{fullntorsion}
Let $K$ be a cyclic quartic extension of $\mathbb{Q}$, and let $E$ be an elliptic curve over $\mathbb{Q}$.  If $E(K)_{\text{tors}} \cong \mathbb{Z}/n\mathbb{Z} \oplus \mathbb{Z}/n\mathbb{Z}$, then $n = 1, 2, 5, \text{ or } 10$.
\end{lemma}
\begin{proof}
Suppose $E(K)_{\text{tors}} \cong \mathbb{Z}/n\mathbb{Z} \oplus \mathbb{Z}/n\mathbb{Z}$.  By the existence of the Weil pairing, full $n$-torsion is defined over a number field $K$ only if the $n^{\text{th}}$ roots of unity are defined over $K$.  In particular we have $\mathbb{Q}(\zeta_{n}) \subseteq K$.  where $\zeta_n$ denotes a primitive $n^{\text{th}}$ root of unity.  Thus we have
$$[K:\mathbb{Q}(\zeta_n)] [\mathbb{Q}(\zeta_n) : \mathbb{Q}] = [K: \mathbb{Q}] = 4$$
and so $[\mathbb{Q}(\zeta_n) : \mathbb{Q}]$ divides 4.  However, notice that $\mathbb{Q}(\zeta_4) = \mathbb{Q}(i)$ and $\mathbb{Q}(\zeta_3) = \mathbb{Q}(\sqrt{-3})$ are not contained in any cyclic quartic number field because of Lemma \ref{quadtotallyreal}.  Therefore, $n$ is not divisible by 3 or 4.  Now examining values of $\varphi(n) = [\mathbb{Q}(\zeta_n):\mathbb{Q}]$ shows that the only possibilities for $n$ are $1, 2, 5, \text{ or } 10$.
\end{proof}

In fact, Bruin and Najman in \cite{bruinnajman} show that it is impossible for full 10-torsion to be defined over a quartic number field (see Theorem \ref{Najman5}).

A vital tool that will be used is the complete classification of $\mathbb{Q}$-rational points on $X_{0}(N)$ for any $N$.  This work was completed due to Fricke, Kenku, Klein, Kubert, Ligozat, Mazur, and Ogg, among others.

\begin{thm}\label{isogoverQ}
If $E/\mathbb{Q}$ has an $n$-isogeny over $\Q$, then $n \leq 19$ or $n \in \{21,25,27,37,43,67,163\}$.  If $E$ does not have complex multiplication, then $n \leq 18$ or $n \in \{21, 25, 37\}$.  
\end{thm}

See \cite{lozanorobledo1}, Theorem 9.5 for a more detailed discussion of this Theorem.

The following lemma helps us understand the 2-torsion of an elliptic curve over $\mathbb{Q}$.

\begin{lemma}\label{2torsionoverQ}
Let $K$ be a number field of degree not divisible by 3, and let $E$ be an elliptic curve over $\mathbb{Q}$.  If $E(\mathbb{Q})[2] = \{\mathcal{O}\}$, then $E(K)[2] = \{\mathcal{O}\}$.
\end{lemma}
\begin{proof}
Suppose $E(\mathbb{Q})[2]$ is trivial.  We may find a model for $E$ in the form $y^2 = f(x)$ for some cubic polynomial $f$.  Since the 2-torsion points of $E$ are all of the form $(\alpha , 0)$ for a root $\alpha$ of $f$, it follows that $f$ is irreducible over $\mathbb{Q}$.  Thus, the field of definition of any 2-torsion point of $E$ is $\mathbb{Q}(\alpha)$ for some root $\alpha$ of $f$, which is a degree 3 number field.  Since $K$ is of degree not divisible by 3, it follows that $\mathbb{Q}(\alpha) \not\subseteq K$ for any root $\alpha$ of $f$, and thus, $E(K)$ contains no non-trivial points of order 2.
\end{proof}

The following Lemma gives a criterion for a point to be halved:

\begin{lemma}[Knapp \cite{knapp}, Theorem 4.2, p. 85]\label{halvingpoint}
Let $K$ be a field of characteristic not equal to 2 or 3, and let $E$ be an elliptic curve over $K$ given by $y^2 = (x-\alpha)(x-\beta)(x-\gamma)$ with $\alpha, \beta, \gamma$ in $K$.  For $P=(x,y) \in E(K)$, there exists a $K$-rational point $Q = (x',y')$ on $E$ such that $[2]Q = P$ if and only if $x- \alpha, x-\beta,$ and $x-\gamma$ are all squares in $K$.  In this case, if we fix the sign of $\sqrt{x-\alpha}, \sqrt{x-\beta},$ and $\sqrt{x-\gamma}$, then $x'$ equals one of the following:
$$\sqrt{x-\alpha}\sqrt{x-\beta} \pm \sqrt{x-\alpha}\sqrt{x-\gamma} \pm \sqrt{x-\beta}\sqrt{x-\gamma} + x$$
or
$$-\sqrt{x-\alpha}\sqrt{x-\beta} \pm \sqrt{x-\alpha}\sqrt{x-\gamma} \mp \sqrt{x-\beta}\sqrt{x-\gamma} + x$$
where the signs are taken simultaneously.
\end{lemma}

The following lemma gives a bound on when full $p$-torsion may appear over a number field.

\begin{lemma}\label{landau} Let $E/\mathbb{Q}$ be an elliptic curve and let $K/ \mathbb{Q}$ be a number field of degree $n$ such that $E$ has full $p$-torsion defined over the Galois closure of $K$.  Then $p -1\leq g(n)$ where $g(n)$ denotes the Landau function.
\end{lemma}
\begin{proof}  Let $L$ be the Galois closure of $K$.  If full $p$-torsion is defined over $E(L)$, then by the existence of the Weil pairing, $\mathbb{Q}(\zeta_{p}) \subseteq L$.  Since $L$ is the Galois closure of a degree $n$ number field, $\Gal(L/\mathbb{Q}) \leq S_{n}$.  By Galois theory we have that $\Gal(\mathbb{Q}(\zeta_{p})/\mathbb{Q}) \cong \left( \mathbb{Z}/p\mathbb{Z} \right)^{\times}$ is a quotient of $\Gal(L/\mathbb{Q})$.  Thus, we must have an element of order $p-1$ in $S_{n}$.  Since the highest order element in $S_{n}$ is given by $g(n)$, it follows that $p-1\leq g(n)$.
\end{proof}

Here is a table of the first couple values for $g(n)$:

$$
\begin{array}{|c|cccccccc|}
\hline n & 1 & 2 & 3 & 4 & 5 & 6 & 7 & 8 \\ \hline
 g(n) & 1 & 2 & 3 & 4 & 6 & 6 & 12 & 15 \\ \hline
\end{array}
$$

The following proposition shows when torsion points over a quartic field $K$ are actually defined over an intermediate quadratic field.

\begin{prop}\label{torsionfromquad} Let $p \equiv 3 \bmod 4$ be a prime with $p \geq 7$.  Let $E/\mathbb{Q}$ be an elliptic curve and let $K / \mathbb{Q}$ be a quartic field such that $E(K)_{\text{tors}}$ contains a point $P$ of order $p$.\\
Then either:
\begin{itemize}
\item $P$ is defined over $\mathbb{Q}$, i.e., $P \in E(\mathbb{Q})[p]$
\item There is $F/\mathbb{Q}$, $F \subseteq K$, $[F:\mathbb{Q}] = 2$ such that $P \in E(F)[p]$.
\end{itemize}
\end{prop}
\begin{proof}
Let $L$ denote the Galois closure of $K$.  Now, consider the Galois representation of $\Gal(\overline{\Q}/\mathbb{Q})$ on $E[p]$ by fixing a basis $\{P,Q\}$ of $E[p]$:
$$\rho : \Gal(\overline{\mathbb{Q}}/\mathbb{Q}) \rightarrow \GL_{2}(\mathbb{F}_{p}).$$
Without loss of generality let $P$ be the point defined over $K$.  Let $\sigma \in \Gal(L/\mathbb{Q})$.  Then $P^{\sigma}= \alpha P + \beta Q$ for some $\alpha, \beta \in \mathbb{F}_p$.  But since $P^{\sigma} - \alpha P= P^{\sigma} + (p-\alpha)P= \beta Q \in E(L)$, it follows that $\beta = 0$ otherwise $E(L)$ would have full $p$-torsion, which is impossible by Lemma \ref{landau} (notice that $g(4)=4$). Thus, we have that $P^{\sigma} \in \langle P \rangle$ for all $\sigma \in \Gal(L/\mathbb{Q})$.  In fact, since $\Gal(L/\mathbb{Q}) \cong \Gal(\overline{\mathbb{Q}}/\mathbb{Q}) / \Gal(\overline{\mathbb{Q}}/L)$ it follows that $$P^{\sigma} \in \langle P \rangle \mbox{ for all } \sigma \in \Gal(\overline{\mathbb{Q}}/\mathbb{Q})$$
Thus, the image of $\rho$ is contained in a Borel subgroup of $\GL_{2}(\mathbb{F}_p)$.  Suppose $$\rho(\sigma)=\begin{pmatrix} \varphi(\sigma) & \tau(\sigma) \\ 0 & \psi(\sigma) \end{pmatrix}$$
where $\varphi , \psi$ are $\mathbb{F}_p$-valued characters of $\Gal(\overline{\mathbb{Q}}/\mathbb{Q})$  and $\tau : \Gal(\overline{\mathbb{Q}}/\mathbb{Q}) \rightarrow \mathbb{F}_{p}$.  Notice that if $P^{\sigma} = aP$, then $\varphi(\sigma) = a$.  Thus, we have that the field of definition $\mathbb{Q}(P) = F$ where $F$ is defined by $\ker(\varphi) = \Gal(\overline{\mathbb{Q}}/F)$.\\\\
Now we claim that $[\mathbb{Q}(P) : \mathbb{Q}] = \# \image \varphi$ which can be seen as follows:
Let $H$ be the subgroup of $G=\Gal(L/\mathbb{Q})$ fixing $\mathbb{Q}(P)$.  Then $$\# \image \varphi = \# \{ P^{\sigma} : \sigma \in G \} = \frac{|G|}{|H|}=[\mathbb{Q}(P):\mathbb{Q}]$$
Now, since $\image\varphi \leq \mathbb{F}_{p}^{\times}$, we have that $$\# \image \varphi \mid p-1.$$ Also since $\mathbb{Q}(P) \subseteq K$ we have that $$\# \image \varphi = [\mathbb{Q}(P) : \mathbb{Q}] \mid [K:Q] = 4.$$ Now, since $p \not\equiv 1 \bmod 4$, it follows that $\# \image \varphi = 1 \mbox{ or } 2$, proving the proposition.
\end{proof}

Notice we can make use of this proposition to narrow the list of possible torsion subgroup structures.\\

\begin{prop} The following subgroups do not appear as torsion subgroups over any quartic number field of any elliptic curve defined over $\mathbb{Q}$:

$$\mathbb{Z} / 14 \mathbb{Z} , \mathbb{Z} / 2 \mathbb{Z} \oplus \mathbb{Z}/ 14 \mathbb{Z}$$
\end{prop}
\begin{proof}
Let $E$ be an elliptic curve over $\mathbb{Q}$, and let $K$ be a quartic number field.  Suppose $E(K) \cong \mathbb{Z} / 14 \mathbb{Z}$.  Suppose further that $E(\mathbb{Q})[2]\neq 0$.  If $E(\mathbb{Q})[7]\neq 0$, then $E$ has a rational torsion point of order 14, impossible by the classification of Mazur.  Thus, since $E(K)[7]\neq 0$, Proposition \ref{torsionfromquad} gives that there exists a quadratic field $F \subseteq K$ such that $E(F)[7] = E(K)[7] \neq 0$.  But $E(\mathbb{Q})[2] \neq 0$ implies that $E(F)[2] \neq 0$ implying that $E(F)[14] \neq 0$ which is impossible by the classification of torsion subgroups of elliptic curves defined over $\mathbb{Q}$ over quadratic number fields.\\
Thus, $E(\mathbb{Q})[2] = 0$.  But this contradicts \ref{2torsionoverQ}, since $E(K)[2] \neq 0$.  %Writing $E$ as $y^2 = f(x)$, this implies that $f(x)$ is irreducible over $\mathbb{Q}$.  Let $P = (a, b) \in E(K)[2]$, then $a$ is a root of $f$ so $[\mathbb{Q}(P) : \mathbb{Q}] \geq \deg(f) = 3$, but $\mathbb{Q}(P) \subseteq K$ which is a degree 4 extension of $\mathbb{Q}$ leading to a contradiction.\\
Suppose $E(K) \cong \mathbb{Z} / 2 \mathbb{Z} \oplus \mathbb{Z} / 14 \mathbb{Z}$.  By the same argument as above we will reach the same contradiction. 
\end{proof}

When the field $K$ is Galois over $\mathbb{Q}$ the following lemma helps to determine when an elliptic curve defined over $\mathbb{Q}$ has an $n$-isogeny defined over $\Q$:

\begin{lemma}\label{isogeny}
Let $K$ be a Galois extension of $\mathbb{Q}$, and let $E$ an elliptic curve over $\mathbb{Q}$.  If $E(K)[n] \cong \mathbb{Z}/n\mathbb{Z}$, then $E$ has an $n$-isogeny over $\mathbb{Q}$.
\end{lemma}
\begin{proof}
Let $\{ P , Q \}$ be a $\mathbb{Z}/n\mathbb{Z}$-basis for $E[n]$.  Without loss of generality we may assume $P \in E(K)$ and $Q \not\in E(K)$. Let $\sigma \in \Gal(\overline{\mathbb{Q}}/\mathbb{Q})$.  Since $K$ is Galois over $\mathbb{Q}$, and $P \in E(K)[n]$, it follows that $P^{\sigma} \in E(K)[n]$.  By assumption, $E(K)[n] = \langle P \rangle$ and thus $P^{\sigma} \in \langle P \rangle$.  Therefore $\langle P \rangle$ is stable under the action of $\Gal(\overline{\mathbb{Q}}/\mathbb{Q})$, which implies $E$ has an $n$-isogeny over $\mathbb{Q}$.
\end{proof}

This combined with the complete classification of rational $n$-isogenies limits the possible torsion subgroups for elliptic curves over $\mathbb{Q}$ in Galois extensions of $\mathbb{Q}$.

\begin{cor}\label{nopoints}
Let $E$ be an elliptic curve over $\mathbb{Q}$, and $K$ a Galois extension of $\mathbb{Q}$.  If $E(K)$ has a point of order $n$= 39, 49, 81, 91, 169, then $E(K)[n]$ is not cyclic.
\end{cor}

If we further limit that $K$ is cyclic quartic, then we obtain the following proposition:

\begin{prop}\label{quarisog} Suppose $K$ is a cyclic quartic number field and $E$ is an elliptic curve over $\mathbb{Q}$.  Suppose $E(K)$ has a point of order $n$ for some odd $n$ relatively prime to 5. Then $E/ \mathbb{Q}$ has an $n$-isogeny over $\Q$.
\end{prop}
\begin{proof}  Choose a basis $\{P,Q\}$ of $E[n]$ such that $P \in E(K)$.  Let $G = \Gal(K / \mathbb{Q}) = \langle \sigma \rangle$.  Then $P^{\sigma} = \alpha P + \beta Q$ for some $\alpha, \beta \in \mathbb{Z} / n \mathbb{Z}$. Then $(n-\alpha) P + P^{\sigma} = \beta Q \in E(K)$ so $\beta = 0$, otherwise $E(K)$ would contain full $l$-torsion for some $l \mid n$, which is impossible since $K$ does not contain any $l^{th}$ roots of unity by Lemma \ref{fullntorsion}.  So, $\langle P \rangle$ is fixed by the action of $G$, and since the action of $\Gal(\overline{\mathbb{Q}} / \mathbb{Q})$ factors through $G$, we have that $\langle P \rangle$ is fixed by the action of the entire group $\Gal(\overline{\mathbb{Q}} / \mathbb{Q})$.  Thus, $E/\mathbb{Q}$ has an $n$-isogeny over $\Q$.  
\end{proof}

The following two theorems are useful in looking at how specifically the torsion can grow from $\mathbb{Q}$ to $F$ to $K$, where $F$ is the intermediate quadratic number field.  See the references given for more detail on them.

\begin{theorem}[Gonz\'{a}lez-Jim\'{e}nez, Tornero \cite{gonztorn}, Theorem 2]\label{phi2g}
Let $\Phi(1)$ be the set of isomorphism classes of torsion subgroups for an elliptic curve over $\mathbb{Q}$.  For a given group $G \in \Phi(1)$, let $\Phi_{\mathbb{Q}}(2,G)$ denote the possible isomorphism classes of $E(F)_{tors}$ for an elliptic curve $E$ such that $E(\mathbb{Q})_{tors} \cong G$ and $F$ a degree 2 extension of $\mathbb{Q}$.  For $G \in \Phi(1)$, the set $\Phi_{\mathbb{Q}}(2,G)$ is the following:
$$\begin{array}{|c|c|}
\hline
G & \Phi_{\mathbb{Q}}(2,G) \\ \hline
\mathcal{C}_1 & \{\mathcal{C}_1, \mathcal{C}_3, \mathcal{C}_5, \mathcal{C}_7, \mathcal{C}_9 \}\\ \hline
\mathcal{C}_2 & \{\mathcal{C}_2, \mathcal{C}_4, \mathcal{C}_6, \mathcal{C}_8, \mathcal{C}_{10}, \mathcal{C}_{12}, \mathcal{C}_{16}, \mathcal{C}_2 \times \mathcal{C}_2, \mathcal{C}_2 \times \mathcal{C}_6, \mathcal{C}_2 \times \mathcal{C}_{10} \} \\ \hline
\mathcal{C}_3 & \{ \mathcal{C}_3, \mathcal{C}_{15}, \mathcal{C}_3 \times \mathcal{C}_3 \} \\ \hline
\mathcal{C}_4 & \{ \mathcal{C}_4, \mathcal{C}_8, \mathcal{C}_{12}, \mathcal{C}_2 \times \mathcal{C}_4, \mathcal{C}_2 \times \mathcal{C}_8, \mathcal{C}_2 \times \mathcal{C}_{12}, \mathcal{C}_4 \times \mathcal{C}_4 \} \\ \hline
\mathcal{C}_5 & \{ \mathcal{C}_5, \mathcal{C}_{15} \} \\ \hline
\mathcal{C}_6 & \{ \mathcal{C}_6, \mathcal{C}_{12}, \mathcal{C}_2 \times \mathcal{C}_6, \mathcal{C}_3 \times \mathcal{C}_6 \} \\ \hline
\mathcal{C}_7 & \{ \mathcal{C}_7 \} \\ \hline
\mathcal{C}_8 & \{ \mathcal{C}_8, \mathcal{C}_{16}, \mathcal{C}_2 \times \mathcal{C}_8 \} \\ \hline
\mathcal{C}_{10} & \{ \mathcal{C}_{10}, \mathcal{C}_2 \times \mathcal{C}_{10} \} \\ \hline
\mathcal{C}_{12} & \{ \mathcal{C}_{12}, \mathcal{C}_2 \times \mathcal{C}_{12} \} \\ \hline
\mathcal{C}_2 \times \mathcal{C}_2 & \{ \mathcal{C}_2 \times \mathcal{C}_2, \mathcal{C}_2 \times \mathcal{C}_4, \mathcal{C}_2 \times \mathcal{C}_6 , \mathcal{C}_2 \times \mathcal{C}_8, \mathcal{C}_2 \times \mathcal{C}_{12} \} \\ \hline
\mathcal{C}_2 \times \mathcal{C}_4 & \{ \mathcal{C}_2 \times \mathcal{C}_4, \mathcal{C}_2 \times \mathcal{C}_8, \mathcal{C}_4 \times \mathcal{C}_4 \} \\ \hline
\mathcal{C}_2 \times \mathcal{C}_6 & \{ \mathcal{C}_2 \times \mathcal{C}_6, \mathcal{C}_2 \times \mathcal{C}_{12} \} \\ \hline
\mathcal{C}_2 \times \mathcal{C}_8 & \{ \mathcal{C}_2 \times \mathcal{C}_8 \} \\ \hline
\end{array}$$
\end{theorem}

The next theorem limits the number of $\mathbb{Q}$-isogenies a particular elliptic curve over $\mathbb{Q}$ can have.  

\begin{theorem}[Kenku \cite{kenku} Theorem 2]\label{8qisogs}
There are at most 8 $\Q$-isomorphism classes of elliptic curves in each $\Q$-isogeny class.
\end{theorem}

In particular, the various combinations of isogenies are classified in that paper.  For example there are only 8 $\mathbb{Q}$-isogeny classes of an elliptic curve if there are eight 2-powered isogenies or four 2-powered isogenies and two 3-powered isogenies (see the reference for more detail).  

The final lemma in this section relate torsion in a quadratic extension to torsion over the base field.

\begin{lemma}\label{twist}
Let $E$ be an elliptic curve defined over $F$, $\alpha$ square-free in $F$, $K = F(\sqrt{\alpha})$.  If $n$ is odd, then there exists an isomorphism
$$
E(K)[n] \cong E(F)[n] \oplus E^{\alpha}(F)[n].
$$
\end{lemma}
\begin{proof}
See Corollary 1.3 (ii) and Lemma 1.4 (ii) in \cite{laskalorenz}
\end{proof}

Finally, in order to examine how the 2-torsion can grow from $\mathbb{Q}$ to a quartic field $K$, the following corollary of Rouse, Zureick-Brown is quite helpful.

\begin{theorem}[Rouse, Zureick-Brown, \cite{2adicimage}, Corollary 1.2]\label{2adicindex}
Let $E$ be an elliptic curve over $\mathbb{Q}$ without complex multiplication.  Then the index of $\rho_{E,2^{\infty}}(\Gal(\overline{\mathbb{Q}}/\mathbb{Q}))$ divides 64 or 96; all such indices occur.  Moreover, the image of $\rho_{E,2^{\infty}}(\Gal(\overline{\mathbb{Q}}/\mathbb{Q}))$ is the inverse image in $\GL_{2}(\mathbb{Z}_2)$ of the image of $\rho_{E,32}(\Gal(\overline{\mathbb{Q}}/\mathbb{Q}))$
\end{theorem}

\section{Torsion over cyclic quartic number fields}\label{cyclicquadproof}

In this section we give the proof of Theorem \ref{cyclicquartic} in a series of propositions and lemmas.  We will first bound the structure of the $p$-Sylow subgroups of $E(K)_{\text{tors}}$ for cyclic quartic number fields $K$ and elliptic curves $E/\mathbb{Q}$.  Then, we will prove certain order torsion points cannot appear in $E(K)$.  Finally, we will provide examples of elliptic curves over $\mathbb{Q}$ that achieve each prescribed torsion subgroup over a cyclic quartic number field $K$.

We narrow down the possible torsion subgroups for $E(K)$ with the next proposition:

\begin{prop}\label{maximalEKtors} Let $E$ be an elliptic curve over $\mathbb{Q}$.  Let $K$ a cyclic quartic extension of $\mathbb{Q}$. Then
$$E(K)_{tors} \subseteq (\mathbb{Z} / 2 \mathbb{Z} \oplus \mathbb{Z} / 16 \mathbb{Z}) \oplus \mathbb{Z} / 9 \mathbb{Z} \oplus (\mathbb{Z} / 5 \mathbb{Z} \oplus \mathbb{Z} / 5 \mathbb{Z}) \oplus \mathbb{Z} / 7 \mathbb{Z} \oplus \mathbb{Z} / 13 \mathbb{Z}.$$
\end{prop}
\begin{proof}
Let $E(K)[p^{\infty}]$ denote the $p$-Sylow subgroup of $E(K)$. From \cite{lozanorobledo1}, we have that $$S_{\mathbb{Q}}(4)= \{2,3,5,7,13 \}$$ where $S_{\mathbb{Q}}(d)$ is the set of primes $p$ for which there exists a number field $K$ of degree $\leq d$ and an elliptic curve $E / \mathbb{Q}$ such that the order of the torsion subgroup of $E(K)$ is divisible by $p$.  Thus, we have $E(K)[p^{\infty}] = \{\mathcal{O}\}$ for all primes not in $S_{\mathbb{Q}}(4)$.  (Recall that $E(\overline{\mathbb{Q}})[m] \cong \mathbb{Z}/m\mathbb{Z} \oplus \mathbb{Z}/m\mathbb{Z}$ for all $m \in \mathbb{N}$.)  We will proceed by examining the $p$-Sylow subgroup for each prime $p = 2,3,5,7,13$.\\
$\underline{p=2}$\\
From Lemma \ref{fullntorsion} we see that full 4-torsion cannot be defined over $K$.  Thus, $E(K)[2^{\infty}] \subseteq \mathbb{Z} / 2 \mathbb{Z} \oplus \mathbb{Z} / 2^{m} \mathbb{Z}$ for some $m$.  Notice that if $E(K)[2^{\infty}] \cong \mathbb{Z} / 2 \mathbb{Z} \oplus \mathbb{Z} / 2^{m} \mathbb{Z}$, then $2E(K)[2^{\infty}] \cong \mathbb{Z} / 2^{m-1} \mathbb{Z}$ and so $E$ has a $\mathbb{Q}$-rational isogeny of degree $2^{m-1}$.  Thus, by Theorem \ref{isogoverQ}, $m-1 \leq 4$, and so $m \leq 5$.  Thus, we have that $E(K)[2^{\infty}] \subseteq  \mathbb{Z} / 2 \mathbb{Z} \oplus  \mathbb{Z} / 32 \mathbb{Z}$.

Now, suppose $E(K)[2^{\infty}] \cong \mathbb{Z}/2\mathbb{Z} \oplus \mathbb{Z}/32\mathbb{Z}$.  Fix a basis $\{P, Q\}$ of $E[32]$ such that $P \in E(K)$.  Let
$$
\rho : \Gal(\overline{\mathbb{Q}}/\mathbb{Q}) \rightarrow \GL(2,\mathbb{Z}/32\mathbb{Z})
$$
be the Galois representation induced by the action of Galois on $E[32]$.  Consider the image of the subgroup $\Gal(\overline{\mathbb{Q}}/K)$ under $\rho$, call this image $H$.  Notice that $P^{\sigma} = P$ for all $\sigma \in H$ since $P$ is defined over $K$.  Further, for $\sigma \in H$, $Q^{\sigma} = a P + b Q$ for some $a,b \in \mathbb{Z}/32\mathbb{Z}$.  Notice however, that $\{16 P , 16 Q \}$ is a basis for $E[2]$ which is contained in $E(K)$.  Thus, for all $\sigma \in H$:
$$
16 Q = (16 Q)^{\sigma} = 16 Q^{\sigma} = 16aP + 16bQ
$$
so we must have that $a \equiv 0 \bmod 2$ and $b \equiv 1 \bmod 2$.  Therefore
$$
H \subseteq \left\{ \begin{pmatrix} 1 & 2x \\ 0 & y \end{pmatrix} : x \in \mathbb{Z}/32\mathbb{Z} , y \in \mathbb{Z}/32\mathbb{Z}^{\times} \right\}.
$$
Denote $\mathcal{H} := \left\{ \begin{pmatrix} 1 & 2x \\ 0 & y \end{pmatrix} : x \in \mathbb{Z}/32\mathbb{Z} , y \in \mathbb{Z}/32\mathbb{Z}^{\times} \right\}$.

Now let us consider the full image of $\rho$, call this image $G$.  Since $E$ has a 16-isogeny over $\Q$, for any $\sigma \in  \Gal(\overline{\mathbb{Q}}/\mathbb{Q})$ it must be that 
$$
\rho(\sigma) \equiv \begin{pmatrix} a & b \\ 0 & d \end{pmatrix} \bmod 16.
$$
In fact, $a,d \in \left(\mathbb{Z}/32\mathbb{Z}\right)^{\times}$ since $E[2] \subseteq E(K)$ and so $$\rho(\sigma) \equiv \begin{pmatrix} 1 & b \\ 0 & 1 \end{pmatrix} \bmod 2.$$
Therefore:
$$
\rho( \Gal(\overline{\mathbb{Q}}/\mathbb{Q}) ) = G \subseteq \left\{ \begin{pmatrix} a & b \\ 16c & d \end{pmatrix} :a,d \in \mathbb{Z}/32\mathbb{Z}^{\times}, b,c \in \mathbb{Z}/32\mathbb{Z} \right\}.
$$
Denote $\mathcal{G} := \left\{ \begin{pmatrix} a & b \\ 16c & d \end{pmatrix} :a,d \in \mathbb{Z}/32\mathbb{Z}^{\times}, b,c \in \mathbb{Z}/32\mathbb{Z} \right\}$.

Suppose there exist an elliptic curve $E$ defined over $\mathbb{Q}$ such that $H \leq \mathcal{H}$ and $G \leq \mathcal{G}$, with $H$ normal in $G$, and $G/H$ isomorphic to $\mathbb{Z}/4\mathbb{Z}$.

If $E/\mathbb{Q}$ does not have CM, then Theorem \ref{2adicindex} shows that the index of $G$ in $\GL(2,\mathbb{Z}/32\mathbb{Z})$ has index dividing 64 or 96.  However,
$$
1536 = [\GL(2,\mathbb{Z}/32\mathbb{Z}) : \mathcal{H}] \leq [\GL(2,\mathbb{Z}/32\mathbb{Z}) : H]
$$
and since $[G:H] = 4$ this gives a lower bound on the index of $G$:
$$[\GL(2,\mathbb{Z}/32\mathbb{Z}) : G] = [GL(2,\mathbb{Z}/32\mathbb{Z}) : H] / [G : H] \geq \frac{1536}{4} = 384$$
and so such a curve does not exist.\\
If $E/\mathbb{Q}$ does have CM, then \cite{cmtorsion} gives a list of all possible isomorphism classes for $E(K)_{\text{tors}}$ for $K$ a quartic field (note that this list is for $E/K$ not necessarily $E/\mathbb{Q}$, and $K$ is any quartic field not necessarily cyclic quartic):
$$
E(K)_{\text{tors}} \in 
\begin{cases} 
\mathbb{Z}/m\mathbb{Z} & \mbox{for } N_1=1, \ldots, 8,10,12,13,21,\\
\mathbb{Z}/2\mathbb{Z} \oplus \mathbb{Z}/2N_2\mathbb{Z} & \mbox{for } N_2=1, \ldots, 5,\\
\mathbb{Z}/3\mathbb{Z} \oplus \mathbb{Z}/3N_3\mathbb{Z} & \mbox{for } N_3=1,2,\\
\mbox{ and } & \mathbb{Z}/4\mathbb{Z} \oplus \mathbb{Z}/4\mathbb{Z}
\end{cases}
$$
and notice that $\mathbb{Z}/2\mathbb{Z} \oplus \mathbb{Z}/32\mathbb{Z}$ is not on that list.

Therefore, no such elliptic curve exists, and so $E(K)[2^{\infty}] \subseteq \mathbb{Z}/2\mathbb{Z} \oplus \mathbb{Z}/16\mathbb{Z}$.\\

$\underline{p=3}$\\
Full 3-torsion cannot be defined over $K$ by Lemma \ref{fullntorsion}.  By Corollary \ref{nopoints} it follows that $E(K)[3^{\infty}] \subseteq \mathbb{Z} / 27 \mathbb{Z}$.  By Lemma \ref{twist}, we have that $E(K)[27] \cong E(F)[27] \oplus E^{\alpha}(F)[27]$ for the quadratic intermediate field $F\subseteq F(\sqrt{\alpha})=K$ for some square-free $\alpha \in F$.  The curve $E^{\alpha}$ may not be defined over $\mathbb{Q}$, but it is certainly defined over $F$.  By the classification of torsion subgroups of elliptic curves defined over quadratics (Theorem \ref{kkmquad}), it follows that neither $E^{\alpha}(F)$ nor $E(F)$ have any points of order 27.  Thus, $E(K)$ does not have any points of order 27, and so $E(K)[3^{\infty}] \subseteq \mathbb{Z} / 9 \mathbb{Z}$.\\

$\underline{p=5}$\\
It is possible for full 5-torsion to be defined over $K$ in the case that $K = \mathbb{Q}(\zeta_{5})$. By Lemma \ref{twist}, we have $E(K)[25] \cong E(F)[25] \oplus E^{\alpha}(F)[25]$ for the quadratic intermediate field $F\subseteq F(\sqrt{\alpha})=K$ for some square-free $\alpha \in F$.  The curve $E^{\alpha}$ may not be defined over $\mathbb{Q}$, but it is certainly defined over $F$.  By the classification of torsion subgroups of elliptic curves defined over quadratics (Theorem \ref{kkmquad}), we have that $E^{\alpha}(F)[25] \cong \mathbb{Z} / 5 \mathbb{Z} \mbox{ or } 0$.  Similarly, $E(F)[25] \cong \mathbb{Z} / 5 \mathbb{Z} \mbox{ or } 0$, and thus $E(K)[5^{\infty}] = E(K)[25] \subseteq \mathbb{Z} / 5 \mathbb{Z} \oplus \mathbb{Z} / 5 \mathbb{Z}$.\\

$\underline{p=7}$\\
Again notice that we cannot have full 7-torsion defined over $K$ by Lemma \ref{fullntorsion}. Now, from Corollary \ref{nopoints}, it follows that $E(K)$ has no points of order 49.  Therefore, $E(K)[7^{\infty}] \subseteq \mathbb{Z} / 7 \mathbb{Z}$.\\

$\underline{p=13}$\\
Again notice that we cannot have full 13-torsion defined over $K$ by Lemma \ref{fullntorsion}.  Again, from Corollary \ref{nopoints}, it follows that $E(K)$ has no points of order 169, so in fact $E(K)[13^{\infty}] \subseteq \mathbb{Z} / 13 \mathbb{Z}$.\\
\end{proof}

We wish to see what torsion subgroups are possible when $\mathbb{Z}/5\mathbb{Z} \oplus \mathbb{Z}/5\mathbb{Z} \subseteq E(K)_{tors}$, but notice this is only possible when $K = \mathbb{Q}(\zeta_5)$.  The following theorem of Bruin and Najman classifies all torsion subgroups that appear over $\mathbb{Q}(\zeta_5)$:

\begin{theorem}[Bruin, Najman, \cite{bruinnajman}, Theorem 5]\label{Najman5}
Let $E$ be an elliptic curve over $\mathbb{Q}(\zeta_5)$.  Then $E(\mathbb{Q}(\zeta_5))_{tors}$ is isomorphic to a subgroup included in Mazur's list over $\mathbb{Q}$ or:
$$
\mathbb{Z}/15\mathbb{Z}, \mathbb{Z}/16\mathbb{Z}, \mathbb{Z}/5\mathbb{Z} \oplus \mathbb{Z}/5\mathbb{Z}.
$$
\end{theorem}

In particular, if $\mathbb{Z}/5\mathbb{Z} \oplus \mathbb{Z}/5\mathbb{Z} \subseteq E(K)_{tors}$ then, in fact, they are equal.  Now we rule out torsion points of certain order.

\begin{lemma}
Let $K$ be a Galois cyclic degree 4 number field, $E$ an elliptic curve over $\mathbb{Q}$. Then $E(K)$ has no points of order 14 or 18.
\end{lemma}
\begin{proof}
Suppose $E$ has a point of order 18 over $K$. By Lemma \ref{2torsionoverQ} it follows that $E$ has a point of order 2 over $\mathbb{Q}$.  From Proposition \ref{maximalEKtors}, we see that $E(K)[18] \cong \mathbb{Z}/2\mathbb{Z} \oplus \mathbb{Z}/18 \mathbb{Z}$ or $\mathbb{Z}/18\mathbb{Z}$.  In either case, $E$ has a 9-isogeny over $\mathbb{Q}$. Let $P$ be a point of order 9 in $E(K)$ and let $\sigma \in  \Gal(K/\mathbb{Q})$.  Since $E(K)[9] = \langle P \rangle$ it follows that $P^{\sigma} = a P$ for some $a \in (\mathbb{Z}/9\mathbb{Z})^{\times}$.  Thus, the orbit of $P$ under $ \Gal(K/\mathbb{Q})$ has size dividing 6, and by the Orbit-Stabilizer Theorem and Galois theory, it follows that $[\mathbb{Q}(P):\mathbb{Q}]$ divides 6.  However, since $K$ is degree 4 over $\mathbb{Q}$, and $\mathbb{Q}(P) \subseteq K$, it follows that the field of definition of $P$ is either 1 or 2.  In either case, we have a point of order 9 and a point of order 2 in a quadratic extension of $\mathbb{Q}$, giving us a point of order 18 in a quadratic extension of $\mathbb{Q}$.  By Theorem \ref{quadcase}, the classification of torsion of elliptic curves defined over $\mathbb{Q}$ over quadratic number fields, we see that this is impossible.\\
A similar proof shows that $E(K)$ has no points of order 14 (notice $\left| (\mathbb{Z}/14\mathbb{Z})^{\times} \right| = 6$).
\end{proof}

\begin{lemma}
Let $K$ be a Galois cyclic degree 4 number field, $E$ an elliptic curve over $\mathbb{Q}$.  Then $E(K)$ has no points of order 21, 26, 35, 39, 40, 45, 48, 63, 65, or 91.
\end{lemma}
\begin{proof} 
Notice that since full 3-torsion and full 13-torsion are not defined over $K$, Corollary \ref{nopoints} shows that it is impossible for $E(K)$ to contain a point of order 39 or 91.  In this proof we will make repeated use Theorem \ref{isogoverQ}.

\underline{$n=21$}\\
If $E(K)$ has a point of order 21, then $E(K)[21] \cong \mathbb{Z}/21\mathbb{Z}$.  Thus, by Lemma \ref{isogeny} $E$ has a $21$-isogeny over $\Q$.  There are only finitely many isomorphism classes of elliptic curves with a $21$-isogeny over $\mathbb{Q}$.  Using division polynomials, we see that none of these curves have an $x$-coordinate of a $21$-torsion point defined in a quadratic or quartic number field.  Thus, it is impossible for $E$ to have $21$-torsion over a quartic number field.

\underline{$n=26$}\\
If $E(K)$ has a point of order 26, then $E(K)[26]$ is isomorphic to either $\mathbb{Z}/26\mathbb{Z}$ or $\mathbb{Z}/2\mathbb{Z} \oplus \mathbb{Z}/26\mathbb{Z}$.  The first case implies the existence of a 26-isogeny over $\mathbb{Q}$, which is impossible.  If $E(K)[26] \cong \mathbb{Z}/2\mathbb{Z} \oplus \mathbb{Z}/26\mathbb{Z}$, then by Lemma \ref{2torsionoverQ} the curve $E$ has a $2$-isogeny over $\mathbb{Q}$.  Further, $E$ has a $13$-isogeny over $\mathbb{Q}$.  This implies there is an elliptic curve over $\mathbb{Q}$ with a $26$-isogeny over $\mathbb{Q}$, which again is impossible.

\underline{$n=35$}\\
If $E(K)$ has a point of order 35, then $E(K)[35]$ is isomorphic to either $\mathbb{Z}/35\mathbb{Z}$ or $\mathbb{Z}/5\mathbb{Z} \oplus \mathbb{Z}/35\mathbb{Z}$.  The first case implies the existence of a 35-isogeny over $\mathbb{Q}$, which is impossible, and the second case is ruled out by \ref{Najman5}.

\underline{$n=40$}\\
If $E(K)$ has a point of order 40, then $E(K)[40]$ is isomorphic to either $\mathbb{Z}/40\mathbb{Z}$ or $\mathbb{Z}/2\mathbb{Z} \oplus \mathbb{Z}/40\mathbb{Z}$.  In either case, this implies a 20-isogeny over $\mathbb{Q}$, which is impossible.

\underline{$n=45$}\\
If $E(K)$ has a point of order 45, then $E(K)[45]$ is isomorphic to either $\mathbb{Z}/45\mathbb{Z}$ or $\mathbb{Z}/5\mathbb{Z} \oplus \mathbb{Z}/45\mathbb{Z}$.  The first case implies the existence of a 45-isogeny over $\mathbb{Q}$, which is impossible, and the second case is ruled out by \ref{Najman5}.

\underline{$n=48$}\\
If $E(K)$ has a point of order 48, then $E(K)[48]$ is isomorphic to either $\mathbb{Z}/48\mathbb{Z}$ or $\mathbb{Z}/2\mathbb{Z} \oplus \mathbb{Z}/48\mathbb{Z}$.  In either case, this implies a 24-isogeny over $\mathbb{Q}$, which is impossible.

\underline{$n=63$}\\
If $E(K)$ has a point of order 63, then $E(K)[63]$ is isomorphic to $\mathbb{Z}/63\mathbb{Z}$, but this implies the existence of a 63-isogeny over $\mathbb{Q}$, which is impossible.

\underline{$n=65$}\\
If $E(K)$ has a point of order 65, then $E(K)[65]$ is isomorphic to either $\mathbb{Z}/65\mathbb{Z}$ or $\mathbb{Z}/5\mathbb{Z} \oplus \mathbb{Z}/65\mathbb{Z}$.  The first case implies the existence of a 65-isogeny over $\mathbb{Q}$, which is impossible, and the second case is ruled out by \ref{Najman5}.

\end{proof}

We separate the next two lemmas from the above because their proofs are more involved.

\begin{lemma}\label{no20}
Let $K$ be a cyclic quartic 4 number field, and let $E$ an elliptic curve over $\mathbb{Q}$.  Then $E(K)$ has no points of order 20. 
\end{lemma}
\begin{proof}
Notice that, as a consequence of Theorem \ref{isogeny} and Lemma \ref{isogoverQ}, it is impossible for $E(K)[20] \cong \mathbb{Z}/20\mathbb{Z}$.  Thus, if $E(K)$ has a point of order 20, then by Proposition \ref{maximalEKtors}, $E(K)_{\text{tors}} \cong \mathbb{Z}/2\mathbb{Z} \oplus \mathbb{Z}/20\mathbb{Z}$.

Let $\mathbb{Q} \subset F \subset K$ denote the chain of number fields in $K$, where $F$ is the unique intermediate quadratic number field.  To prove that $\mathbb{Z}/2\mathbb{Z} \oplus \mathbb{Z}/20\mathbb{Z}$ torsion is impossible, we will break up the argument into different cases based on $E(F)_{\text{tors}}$ and subsequently $E(\mathbb{Q})_{\text{tors}}$.

Suppose $E(K)_{\text{tors}} \cong \mathbb{Z}/2\mathbb{Z} \oplus \mathbb{Z}/20\mathbb{Z}$.  Notice, since $E(K)[5] \cong \mathbb{Z}/5\mathbb{Z}$ we have by Lemma \ref{isogeny} that $E$ has a 5-isogeny over $\mathbb{Q}$.  Kenku's classification of $\mathbb{Q}$-isomorphism classes of elliptic curves (Theorem \ref{8qisogs}) states that if $E$ has a 5-isogeny over $\mathbb{Q}$, then $E$ has at most one $\mathbb{Q}$-rational subgroup of order 2.  Therefore, $E(\mathbb{Q})[2] \subseteq \mathbb{Z}/2\mathbb{Z}$.  Further, since $E(K)[2] \neq \{ \mathcal{O} \}$,   Lemma \ref{2torsionoverQ} gives that in fact $E(\mathbb{Q})[2] \cong \mathbb{Z}/2\mathbb{Z}$.  Thus given that $E(K)_{\text{tors}} \cong \mathbb{Z}/2\mathbb{Z} \oplus \mathbb{Z}/20\mathbb{Z}$ we have that $E(\mathbb{Q})_{\text{tors}} \cong \mathbb{Z}/2\mathbb{Z}$ or $\mathbb{Z}/10\mathbb{Z}$.

Theorem \ref{phi2g} classifies precisely how certain torsion over $\mathbb{Q}$ can grow in a quadratic extension of $\mathbb{Q}$.  So given that $E(K)_{\text{tors}} \cong \mathbb{Z}/2\mathbb{Z} \oplus \mathbb{Z}/20\mathbb{Z}$, we have the following five possibilities for $E(F)_{\text{tors}}$:
$$\mathbb{Z}/2\mathbb{Z},\ \  \mathbb{Z}/4\mathbb{Z},\ \  \mathbb{Z}/10\mathbb{Z},\ \  \mathbb{Z}/2\mathbb{Z} \oplus \mathbb{Z}/2\mathbb{Z},\ \  \mathbb{Z}/2\mathbb{Z} \oplus \mathbb{Z}/10\mathbb{Z}.$$
Notice that if full 2-torsion is not defined over $F$, then the field extension adjoining the second 2-torsion point is $K=F(\sqrt{\triangle_{E}})$, which is a biquadratic extension, not a cyclic quartic extension.  Therefore we need only consider when $F = \mathbb{Q}(\sqrt{\triangle_{E}})$ and
$$E(F)_{\text{tors}} \cong \mathbb{Z}/2\mathbb{Z} \oplus \mathbb{Z}/2\mathbb{Z} \mbox{ or } \mathbb{Z}/2\mathbb{Z} \oplus \mathbb{Z}/10\mathbb{Z}.$$
In either case, $E(F)[2^{\infty}] \cong \mathbb{Z}/2\mathbb{Z} \oplus \mathbb{Z}/2\mathbb{Z}$, and so $E$ gains a point of order 4 in $K$.  Since $E(\mathbb{Q})[2] \cong \mathbb{Z}/2\mathbb{Z}$, there exists a model for $E$ of the form
$$E: y^2 = x(x^2+bx+c).$$
Suppose first that the point of order 4 is $Q$, and $2Q = P = (0,0)$.  By Lemma \ref{halvingpoint}, we have that the following are squares in $K$:
$$0, \frac{b+\sqrt{b^2-4c}}{2}, \frac{b-\sqrt{b^2-4c}}{2}.$$
That is, $K = F(\sqrt{\frac{b+\sqrt{b^2-4c}}{2}}, \sqrt{\frac{b-\sqrt{b^2-4c}}{2}}) = F(\sqrt{2b+2\sqrt{m}},\sqrt{2b-2\sqrt{m}})$ where $m = b^2-4c \neq 0$ since $\triangle_{E} \neq 0$.  By Lemma \ref{criteriaforcyclic} we must have $$\frac{(2b)^2}{m} - 2^2 = 1$$ and so $(2b)^2 = 5m$. Thus $m = 2^2 \cdot 5 \cdot t^2$ for some $t \in \mathbb{Q}$ and so $b=5t$.  Further, $m=2^2 \cdot 5 \cdot t^2$ and so $25t^2 - 4c = 20t^2$.  Thus, $4c = 5t^2$.  Substituting $t = 2s$ gives $b = 10s$ and $c=5s^2$.  That is, we obtain a 1-parameter family of elliptic curves such that the 2-power-torsion grows:  
$$\mathbb{Z}/2\mathbb{Z} \rightarrow \mathbb{Z}/2\mathbb{Z} \oplus \mathbb{Z}/2\mathbb{Z} \rightarrow \mathbb{Z}/2\mathbb{Z} \oplus \mathbb{Z}/4\mathbb{Z}$$
 over $\mathbb{Q} \rightarrow F \rightarrow K$.  However, all the curves in this 1-parameter family have the same $j$-invariant, namely $j(E) = 78608$, and thus are all isomorphic over $\mathbb{Q}$.  Now, if one of them acquired $\mathbb{Z}/2\mathbb{Z} \oplus \mathbb{Z}/20\mathbb{Z}$ over a cyclic quartic, then, in particular, there would be a 5-isogeny over $\Q$ for said curve.  But 5-isogenies is an isomorphism invariant property, and so the entire family of curves would posses a 5-isogeny over $\Q$. However, we can check that $\Phi_{5}(X, 78608)$ has no rational roots, where $\Phi_{5}(X,Y)$ is the $5^{\text{th}}$ classical modular polynomial, so this family does not have a 5-isogeny over $\Q$.

Now, suppose one of the other 2-torsion points is halved.  Then, $2Q= P = ( \frac{-b\pm\sqrt{b^2 - 4c}}{2},0)$.  By Lemma \ref{halvingpoint} we have that the following are squares in $K$:
$$ \frac{-b\pm\sqrt{b^2 - 4c}}{2}, \pm \sqrt{b^2-4c}.$$
However, it is easy to see by Lemma \ref{criteriaforcyclic} that the field $K = F( \sqrt{\pm \sqrt{b^2-4c}})$ is not a cyclic quartic number field.
Therefore, it is impossible for $E$ to have torsion subgroup isomorphic to $\mathbb{Z}/2\mathbb{Z} \oplus \mathbb{Z}/20\mathbb{Z}$ over a cyclic quartic number field.
\end{proof}

\begin{lemma}
Let $K$ be a Galois cyclic degree 4 number field, $E$ an elliptic curve over $\mathbb{Q}$.  Then $E(K)$ has no points of order 24. 
\end{lemma}
\begin{proof}
Notice that by a simple consequence of Theorem \ref{isogeny} and Lemma \ref{isogoverQ} it is impossible for $E(K)[24] \cong \mathbb{Z}/24\mathbb{Z}$.  Thus, if $E(K)$ has a point of order 24, then by Proposition \ref{maximalEKtors}, $E(K)_{\text{tors}} \cong \mathbb{Z}/2\mathbb{Z} \oplus \mathbb{Z}/24\mathbb{Z}$.

Let $\mathbb{Q} \subset F \subset K$ denote the chain of number fields in $K$, where $F$ is the unique intermediate quadratic number field.  To prove that $\mathbb{Z}/2\mathbb{Z} \oplus \mathbb{Z}/24\mathbb{Z}$ torsion is impossible, we will break up the argument into different cases based on $E(F)_{\text{tors}}$ and subsequently $E(\mathbb{Q})_{\text{tors}}$.

Suppose $E(K)_{\text{tors}} \cong \mathbb{Z}/2\mathbb{Z} \oplus \mathbb{Z}/24\mathbb{Z}$.  Notice, since $E(K)[3] \cong \mathbb{Z}/3\mathbb{Z}$ we have by Lemma \ref{isogeny} that $E$ has a 3-isogeny over $\mathbb{Q}$.  Kenku's classification of $\mathbb{Q}$-isomorphism classes of elliptic curves (Theorem \ref{8qisogs}) states that if $E$ has a 3-isogeny over $\mathbb{Q}$, then $E$ has at most four $\mathbb{Q}$-rational subgroups of order a power of 2 (including $\{\mathcal{O}\}$).

Further, given that $E(K)_{\text{tors}} \cong \mathbb{Z}/2\mathbb{Z} \oplus \mathbb{Z}/24\mathbb{Z}$ it follows that
$$E(\mathbb{Q})[2^{\infty}] \cong 
\begin{cases}
\{ \mathcal{O} \}, \\
\mathbb{Z}/2\mathbb{Z}, \\
\mathbb{Z}/4\mathbb{Z}, \\
\mathbb{Z}/8\mathbb{Z}, \\
\mathbb{Z}/2\mathbb{Z} \oplus \mathbb{Z}/2\mathbb{Z}, \\
\mathbb{Z}/2\mathbb{Z} \oplus \mathbb{Z}/4\mathbb{Z}, \\
\mathbb{Z}/2\mathbb{Z} \oplus \mathbb{Z}/8\mathbb{Z}.
\end{cases}$$
We can rule out  $\mathbb{Z}/2\mathbb{Z} \oplus \mathbb{Z}/4\mathbb{Z}$ because it has six subgroups of order a power of 2 and $\mathbb{Z}/2\mathbb{Z} \oplus \mathbb{Z}/8\mathbb{Z}$ because it has eight.  We can further rule of $E(\mathbb{Q})[2^{\infty}] \cong \mathbb{Z}/8\mathbb{Z}$ since then $E$ would have an 8-isogeny over $\mathbb{Q}$, and thus a 24-isogeny over $\mathbb{Q}$, which is impossible.  We can rule out $\mathbb{Z}/2\mathbb{Z} \oplus \mathbb{Z}/2\mathbb{Z}$ because this yields four cyclic-subgroups defined over $\mathbb{Q}$ of 2-power order (three of order 2 and the trivial subgroup), but since we know $E(K)[2^{\infty}] \cong \mathbb{Z}/2\mathbb{Z} \oplus \mathbb{Z}/8\mathbb{Z}$, it follows that $E$ has a 4-isogeny over $\mathbb{Q}$, and therefore at least one more cyclic-subgroup defined over $\mathbb{Q}$.  Finally, Lemma \ref{2torsionoverQ} rules out $E(\mathbb{Q})[2^{\infty}] \cong \{\mathcal{O}\}$.  Thus, 

$$E(\mathbb{Q})[2^{\infty}] \cong 
\begin{cases}
\mathbb{Z}/2\mathbb{Z}, \\
\mathbb{Z}/4\mathbb{Z}.
\end{cases}$$

Theorem \ref{phi2g} classifies precisely how certain torsion over $\mathbb{Q}$ can grow in a quadratic extension of $\mathbb{Q}$.  This gives the following possibilities for $E(F)[2^{\infty}]$:
$$ \mathbb{Z}/2\mathbb{Z},\ \ \mathbb{Z}/4\mathbb{Z},\ \ \mathbb{Z}/8\mathbb{Z},\ \ \mathbb{Z}/2\mathbb{Z} \oplus \mathbb{Z}/2\mathbb{Z},\ \ \mathbb{Z}/2\mathbb{Z} \oplus \mathbb{Z}/4\mathbb{Z},\ \ \mathbb{Z}/2\mathbb{Z} \oplus \mathbb{Z}/8\mathbb{Z}.$$
Again, notice that if full 2-torsion is not defined over $F$, then $K= F(\sqrt{\triangle_{E}})$, and so $K$ is a biquadratic number field, not a cyclic quartic number field.  Thus we need only consider
$$E(F)[2^{\infty}] \cong \begin{cases} \mathbb{Z}/2\mathbb{Z} \oplus \mathbb{Z}/2\mathbb{Z}, \\ \mathbb{Z}/2\mathbb{Z} \oplus \mathbb{Z}/4\mathbb{Z}, \\ \mathbb{Z}/2\mathbb{Z} \oplus \mathbb{Z}/8\mathbb{Z} \end{cases}.$$

\textbf{Case 1:} Suppose $E(F)[2^{\infty}] \cong \mathbb{Z}/2\mathbb{Z} \oplus \mathbb{Z}/2\mathbb{Z}$.  Then it must follow that $E(\mathbb{Q})[2^{\infty}] \cong \mathbb{Z}/2\mathbb{Z}$ , and so we have a model for E of the form $$y^2 = x(x^2+bx+c).$$  Following an identical argument to that of Lemma \ref{no20}, it must be that the point that is halved in $K$ is (0,0).  We again find a 1-parameter family of curves with torsion $E(\mathbb{Q})[2^{\infty}] \cong \mathbb{Z}/2\mathbb{Z}$ and $E(F)[2^{\infty}] \cong \mathbb{Z}/2\mathbb{Z} \oplus \mathbb{Z}/2\mathbb{Z}$ that gains a point of order 4 in a cyclic quartic extension $K$ all with $j$-invariant 78608.  We check that $\Phi_3(X, 78608)$ has no rational roots, where $\Phi_3(X,Y)$ is the $3^{\text{rd}}$ classical modular polynomial, and therefore these curves have no $3$-isogeny over $\Q$.\\

%\textcolor{red}{This feels bad, since j = 78608 doesn't actually have C2xC8 torsion over the quartic.  Can we instead rule it out another way?}\\

%\textcolor{red}{Consider the following argument:
%Now, suppose one of the 2-torsion points other than (0,0) is halved.  Then, $2Q= P = ( \frac{-b\pm\sqrt{b^2 - 4c}}{2},0)$.  By Lemma \ref{halvingpoint} we have that the following are squares in $K$:
%$$ \frac{-b\pm\sqrt{b^2 - 4c}}{2}, \pm \sqrt{b^2-4c}.$$
%However, it is easy to see by Lemma \ref{criteriaforcyclic} that the field $K = F( \sqrt{\pm \sqrt{b^2-4c}})$ is not a cyclic quartic number field.
%Suppose that $(0,0)$ is the point that is halved.  Then $2Q= P = (0,0)$.  So Lemma \ref{halvingpoint} gives that $K = F(\sqrt{2b+2\sqrt{m}}, \sqrt{2b-2\sqrt{m}})$, where $m = b^2-4c$.  Further, Lemma \ref{halvingpoint} gives the precise $x$-coordinate of the points of order 4, $Q$ and $3Q$, which is one of $\pm\sqrt{c}$.  Since one of these points must also be halved to get a point of order 8, a repeated use of Lemma \ref{halvingpoint} gives that $\sqrt[4]{c}$ or $\sqrt{i\sqrt{c}}$ are in $K$, both of which contradicting that $K$ is a cyclic quartic number field.}\\

%\textcolor{red}{What worries me about this argument is that maybe a different point of order 4 is halved, such as $R+Q$ or $R+3Q$ for $R$ one of the other points of order 2.}\\

\textbf{Case 2:} Suppose $E(F)[2^{\infty}] \cong \mathbb{Z}/2\mathbb{Z} \oplus \mathbb{Z}/4\mathbb{Z}$.  From Theorem \ref{phi2g} we have $E(\mathbb{Q})[2^{\infty}] \cong \mathbb{Z}/4\mathbb{Z}$.

%If $E(\mathbb{Q})[2^{\infty}] \cong \mathbb{Z}/2\mathbb{Z} \oplus \mathbb{Z}/2\mathbb{Z}$, then we have a model for $E$ of the form: $$y^2 = x(x+M)(x+N)$$ for $M,N \in \mathbb{Q}$ not both squares.  Without loss of generality we may assume our point of order 4 in $F$ is $Q$ such that $2Q = P=(0,0)$.  Then by CRITERIA FOR HALVING POINTS we have that $F = \mathbb{Q}(\sqrt{M},\sqrt{N})$ and $Q = (\pm\sqrt{MN},y_0)$.  In $K$, a point of order 4 is halved, so again by our criteria, we have that $\pm\sqrt{MN}$ is square in $K$, i.e. $\sqrt{\pm\sqrt{MN}} \in K$, contradicting that $K$ is a cyclic quartic number field.

Since $E(\mathbb{Q})[2^{\infty}] \cong \mathbb{Z}/4\mathbb{Z}$, we have a model for $E$ of the form: $$y^2 = x(x^2+bx+c)$$ and $F = \mathbb{Q}(\sqrt{m})$, where $m=b^2-4c$, and $E(\mathbb{Q})[4] = \{ \mathcal{O}, (\sqrt{c}, y'), (0,0), (-\sqrt{c},y'') \}$.  Let $P = (d, y_0) \in E(\mathbb{Q})$ be the point of order 4 that is halved in $K$ ($d = \sqrt{c}$ or $-\sqrt{c}$).  Then by lemma \ref{halvingpoint} it follows that $K = F(\sqrt{d}, \sqrt{2d+2b-2\sqrt{m}}, \sqrt{2d+2b+2\sqrt{m}})$.  Now using lemma \ref{criteriaforcyclic}, $K$ is quartic cyclic if and only if
$$\frac{(2d+2b)^2}{m}-4 = 1$$
We see that $(2d+2b)^2 =5m$ so $4(d+b)^2 = 5m$ and thus $m = 2^2 \cdot 5 \cdot t^2$ for some $t \in \mathbb{Q}$.  Thus $b^2-4c = 20t^2$.  Further, $d+b = 5t$, and so $b = 5t-d = 5t \pm \sqrt{c}$, and so $(5t\pm\sqrt{c})^2 - 4c = 20t^2$ simplifying to $$3c\pm10t\sqrt{c}-5t^2 = 0$$ or $$3d^2 \pm 10td - 5t^2 = 0.$$  Solving for $d$ yields:
$$d = \frac{\pm 5t \pm 2t\sqrt{10}}{3}.$$
Further, $d \neq 0$, since otherwise $t=0$, implying $\triangle_{E} = 0$. So, this contradicts that $d \in \mathbb{Q}$.  Therefore, it is impossible for $E(F)[2^{\infty}] \cong \mathbb{Z}/2\mathbb{Z} \oplus \mathbb{Z}/4\mathbb{Z}$.\\

\textbf{Case 3:} Suppose $E(F)[2^{\infty}] \cong \mathbb{Z}/2\mathbb{Z} \oplus \mathbb{Z}/8\mathbb{Z}$.  From Theorem \ref{phi2g} we have $E(\mathbb{Q})[2^{\infty}] \cong \mathbb{Z}/4\mathbb{Z}$.

Notice that since full 2-torsion is defined over $F$ but not $\mathbb{Q}$ that $F = \mathbb{Q}(\triangle_{E})$ where $\triangle_{E}$ is the discriminant of $E$.  We may search the Rouse, Zureich-Brown database \cite{2adicimage} for all families of elliptic curves (without CM) whose 2-adic image satisfies $E(\mathbb{Q})[2^{\infty}] \cong \mathbb{Z}/4\mathbb{Z}$ and $E(\mathbb{Q}(\triangle_{E}))[2^{\infty}] \cong \mathbb{Z}/2\mathbb{Z} \oplus \mathbb{Z}/8\mathbb{Z}$.  These families are:
$$\text{X102k, X195h, X197f, X202e, X213k, X223b, X230h, X235f}$$
Now, taking a representative from these families and computing the isogenies of the curve shows that all of the curves in each of these families has an 8-isogeny over $\mathbb{Q}$.  But since we are assuming $E(K)[3] \cong \mathbb{Z}/3\mathbb{Z}$, these curves have both an 8-isogeny and a 3-isogeny over $\mathbb{Q}$, thus giving a curve with a 24-isogeny over $\mathbb{Q}$, which is impossible by Theorem \ref{isogoverQ}.

If $E/\mathbb{Q}$ does have CM, then \cite{cmtorsion} gives a list of all possible isomorphism classes for $E(K)_{\text{tors}}$ for $K$ a quartic field (note that this list is for $E/K$ not necessarily $E/\mathbb{Q}$, and $K$ is any quartic field not necessarily cyclic quartic):
$$
E(K)_{\text{tors}} \in 
\begin{cases} 
\mathbb{Z}/m\mathbb{Z} & \mbox{for } N_1=1, \ldots, 8,10,12,13,21,\\
\mathbb{Z}/2\mathbb{Z} \oplus \mathbb{Z}/2N_2\mathbb{Z} & \mbox{for } N_2=1, \ldots, 5,\\
\mathbb{Z}/3\mathbb{Z} \oplus \mathbb{Z}/3N_3\mathbb{Z} & \mbox{for } N_3=1,2,\\
\mbox{ and } & \mathbb{Z}/4\mathbb{Z} \oplus \mathbb{Z}/4\mathbb{Z},
\end{cases}
$$
and notice that none of these groups have a point of order 24.

Therefore, it is impossible for an elliptic curve defined over $\mathbb{Q}$ to have a point of order 24 in a cyclic quartic extension of $\mathbb{Q}$.
%Now, consider the Galois representation of $\Gal(L/\mathbb{Q})$ on $E[8]$ by fixing a basis $\{P,Q\}$ of $E[8]$:
%$$\rho : \Gal(\overline{\mathbb{Q}}/\mathbb{Q}) \rightarrow \GL_{2}(\mathbb{Z}/8\mathbb{Z}).$$
%Given that $E(\mathbb{Q})[2^{\infty}] \cong \mathbb{Z}/4\mathbb{Z}$, there exists a basis $\{P,Q\}$ of $E[8]$ such that the image of $\rho$ mod 4 is contained in a Borel subgroup of $\GL_{2}(\mathbb{Z}/4\mathbb{Z})$.  Further, 

\end{proof}

Thus, the possible isomorphism classes for $E(K)_{tors}$ for an elliptic curve over $\mathbb{Q}$ and a cyclic quartic number field $K$ is a subset of 
$$
\begin{array}{ll}
\mathbb{Z} / N_1 \mathbb{Z}, & N_1 = 1, \ldots ,16 , N_1 \neq 11, 14,\\
\mathbb{Z} / 2\mathbb{Z} \oplus \mathbb{Z}/2N_2 \mathbb{Z}, & N_2 = 1, \ldots , 6, 8,\\
\mathbb{Z} / 5\mathbb{Z} \oplus \mathbb{Z}/ 5\mathbb{Z}. & \\
\end{array}
$$
and the examples given in the next section show that precisely all of these groups occur.

%\textcolor{red}{Make a remark about why there is always a cyclic quartic extension that does not add any torsion}

\section{Examples}\label{examples}

In this section we provide examples of an elliptic curve $E$ defined over $\mathbb{Q}$ with prescribed torsion over a quartic Galolis number field.

The first table gives an elliptic curve $E$ defined over $\mathbb{Q}$ and a polynomial $f$ such that $K = \mathbb{Q}(\alpha)$ is a cyclic quartic number field for $\alpha$ a root of $f$, and $E(K)_{tors} \cong G$ for each group $G$ appearing in Theorem \ref{cyclicquartic} not appearing in Mazur's list completing the proof of Theorem \ref{cyclicquartic}.

\begin{center}
\begin{tabular}{|c|c|c|}\hline
$G$ & $E$ & $f(x)$ \\ \hline  &  &  \\
$\mathbb{Z}/2\mathbb{Z} \oplus \mathbb{Z}/10 \mathbb{Z}$ & \small{[1,0,0,-828,9072]}& \small{$x^4 - 20x^2 + 10$} \\[5pt]
$\mathbb{Z}/2\mathbb{Z} \oplus \mathbb{Z}/12 \mathbb{Z}$ & \small{[1,0,0,-5557266,-3547208700]} & \small{$13x^4 - 26x^2 + 4$} \\[5pt]
$\mathbb{Z}/2\mathbb{Z} \oplus \mathbb{Z}/16 \mathbb{Z}$ & \small{[1,1,1,-5,2]}, $15a3$ & \small{$x^4-3x^3-6x^2+18x-9$} \\[5pt]
$\mathbb{Z}/13 \mathbb{Z}$ & \small{[1, 0, 1, 266982, 42637516]} & \tiny{$x^4+5688x^3-187682160510x^2+$}\\ & & \tiny{$3933601303888x-6842546623573620597$} \\[5pt] 
$\mathbb{Z}/15 \mathbb{Z}$ & \small{[1,1,1,-3,1]}, $50b1$ & $x^4-10x^2+5$ \\ [5pt] 
$\mathbb{Z}/16 \mathbb{Z}$ & \small{[1,1,1,0,0]}, $15a8$ & $x^4 + 3x^3 + 4x^2 + 2x + 1$\\  [5pt] 
$\mathbb{Z}/5 \mathbb{Z} \oplus \mathbb{Z}/5\mathbb{Z}$ & \small{[0,-1,1,-10,-20]}, $11a1$ & $x^4 + x^3 + x^2 + x + 1$\\ [5pt] \hline
\end{tabular}
\end{center}

Now we finish the proof of Theorem \ref{quarticgalois} by proving that each subgroup listed, except for $\mathbb{Z}/15\mathbb{Z}$ appears for infinitely many non-isomorphic rational elliptic curves over some quartic Galois number field.

First, notice that, by Lemma \ref{biquadaddnotorsion}, any group appearing infinitely often over a quadratic number field, as in Theorem \ref{quadcase}, must also appear infinitely often over quartic Galois number fields.

The following two theorems of Jeon, Kim, Lee give infinite families of elliptic curves defined over $\mathbb{Q}$ that have torsion subgroup $\mathbb{Z}/4\mathbb{Z} \oplus \mathbb{Z}/8\mathbb{Z}$ and $\mathbb{Z}/6\mathbb{Z} \oplus \mathbb{Z}/6\mathbb{Z}$, respectively, over a biquadratic field.

\begin{theorem}[\cite{jeonkimlee}, Theorem 4.10]
Let $K = \mathbb{Q}(\sqrt{-1},\sqrt{t^4-6t^2+1})$ with $t \in \mathbb{Q}$ and $t \neq 0, \pm 1$, and let $E$ be an elliptic curve defined by the equation
$$y^2 + xy - (\nu^2 - \frac{1}{16})y = x^3 - (\nu^2 - \frac{1}{16})x^3,$$
where $\nu = \frac{t^4 - 6t^2 + 1}{4(t^2+1)^2}$.  Then the torsion subgroup of $E$ over $K$ is equal to $\mathbb{Z}/4\mathbb{Z} \oplus \mathbb{Z}/8\mathbb{Z}$ for almost all $t$.
\end{theorem}

\begin{theorem}[\cite{jeonkimlee}, Theorem 4.11]
Let $K = \mathbb{Q}(\sqrt{-3},\sqrt{8t^3+1})$ with $t \in \mathbb{Q}$ and $t \neq 0, 1, -\frac{1}{2}$, and let $E$ be an elliptic curve defined by the equation
$$y^2 = x^3 - 27\mu(\mu^3 + 8)x + 54(\mu^6 - 20\mu^3 - 9),$$
where $\mu = \frac{2t^3+1}{3t^2}$.  Then the torsion subgroup of $E$ over $K$ is equal to $\mathbb{Z}/6\mathbb{Z} \oplus \mathbb{Z}/6\mathbb{Z}$ for almost all $t$.
\end{theorem}

We now prove that there are infinitely many elliptic curves defined over $\mathbb{Q}$ with non-trivial 13-torsion over some cyclic quartic field.

\begin{proposition}
There exists infinitely many elliptic curves $E/\mathbb{Q}$ such that there exists a cyclic quartic field $K$ such that $E(K)$ has a 13-torsion point.
\end{proposition}
\begin{proof}
Let $\Delta = \{ \pm 1, \pm 5 \} \subseteq (\mathbb{Z}/13\mathbb{Z})^{\times}$.  Let $X_{\Delta}(13)$ be the modular curve defined over $\Q$ associated to the congruence subgroup:
$$
\Gamma_{\Delta}(13) := \left\{ \begin{pmatrix} a & b \\ c & d \end{pmatrix} \in \SL_{2}(\Z) \middle| \; \; a \bmod 13 \in \Delta, 13 \mid c \right\}
$$
Then by \cite{jeonkim} (Theorem 1.1 and Table 1), $X_{\Delta}(13)$ has genus 0.  The example of the curve achieving torsion subgroup $\mathbb{Z}/13\mathbb{Z}$ over a cyclic quartic field in the Table above shows that $X_{\Delta}(13)(\mathbb{Q})$ (and hence $Y_{\Delta}(13)(\mathbb{Q})$) has infinitely many points.
Now let $(E, \langle P \rangle )$, where $E/\mathbb{Q}$ and $\langle P \rangle$ is a 13-cycle of $E$, be a point on $X_{\Delta}(13)(\mathbb{Q})$.  Consider the Galois representation
$$
\rho : \Gal(\overline{\mathbb{Q}}/\mathbb{Q}) \rightarrow \GL(2,\mathbb{Z}/13\mathbb{Z})
$$
induced by the action of Galois on a basis $\{ P, R\}$ of $E[13]$.  Then we have, for any $\sigma \in \Gal(\overline{\mathbb{Q}}/\mathbb{Q})$,
$$
\rho(\sigma) = \begin{pmatrix} \varphi(\sigma) & * \\ 0 & * \end{pmatrix}
$$
Let $K$ be the field of definition of $P$.  Notice that $\ker \varphi = \Gal(\overline{\mathbb{Q}}/K)$, and so in particular we have $\text{Image } \varphi \cong \Gal(K/\mathbb{Q})$. By our choice of $\Delta$ it must be that for any $\sigma \in \Gal(\overline{\mathbb{Q}}/\mathbb{Q})$,
$$P^{\sigma} \in \{ \pm P, \pm 5 P \}.$$
Thus, $\sigma$ acts on $\langle P \rangle$ by multiplication by an element of order dividing 4.  In fact, there must exist a $\sigma \in \Gal(\overline{\mathbb{Q}}/\mathbb{Q})$ that acts on $\langle P \rangle$ by an element of order exactly 4, since otherwise $\text{Image } \varphi$ has size 2 or 1, implying that the degree of $K$ over $\mathbb{Q}$ is 2 or 1.  However, by Theorem \ref{quadcase} it is impossible for an elliptic curve over $\mathbb{Q}$ to have a 13-torsion point defined over a quadratic number field.  Thus, $\Gal(K/\mathbb{Q}) \cong \text{Image } \varphi \cong \{ \pm 1, \pm 5\} \cong \mathbb{Z}/4\mathbb{Z}$, and the proposition is complete.
\end{proof}

The group $\mathbb{Z}/5\mathbb{Z} \oplus \mathbb{Z}/5\mathbb{Z}$ appears for infinitely many non-isomorphic curves over the cyclic quartic field $\mathbb{Q}(\zeta_5)$, shown in Theorem 1.1 of \cite{gonzlozano}.

Finally we show an infinite family of elliptic curves such that there exists a biquadratic number field $K$ with $E(K)_{\text{tors}} \cong \mathbb{Z}/2\mathbb{Z}\oplus\mathbb{Z}/16\mathbb{Z}$ using the construction given in \cite{fujita} Section 5.

\begin{proposition}
Let $E$ be the elliptic curve given by $y^2 = x(x+(t^2-1)^4)(x+(2t)^4)$, where $t>1$ is an integer, and $K = \mathbb{Q}(\sqrt{t(t^2-1)},\sqrt{(t^2-1)(t^2+1)(t^2+2t-1)})$.  Then $E(K)_{\text{tors}} \cong \mathbb{Z}/2\mathbb{Z} \oplus \mathbb{Z}/16\mathbb{Z}$.  
\end{proposition}
\begin{proof}
The curve $E$ fulfills the criteria put forth in \cite{fujita} Section 5 Case 1.(II) with $u = t^2 -1$, $v = 2t$, and $w = t^2+1$.  Notice that $\sqrt{-1} \not\in K$ because, since $t >1$, $K$ is totally real.  Therefore $E(K)_{\text{tors}} \cong \mathbb{Z}/2\mathbb{Z} \oplus \mathbb{Z}/16\mathbb{Z}$.
\end{proof}

Finally it remains to be remarked as to why $\mathbb{Z}/15\mathbb{Z}$ only appears finitely often as the torsion subgroup of rational elliptic curves over quartic Galois number fields.  This follows from Lemma \ref{isogeny} and the fact that there are only finitely many $j$-invariants of elliptic curves over $\mathbb{Q}$ having a $\Q$-rational15-isogeny.  

%%%%%%%%%%%%%%%%%%%%%%%%%%%%%%%%%%%%%%%%%%%%%%%%%%%%%%%%%%%%%%%%%%%%%%%%%%%%%%%%
%%%%%%%%%%%%%%%%%%%%%%%%%%%%%%%%%%%%%%%%%%%%%%%%%%%%%%%%%%%%%%%%%%%%%%%%%%%%%%%%
%%%%%%%%%%%%%%%%%%%%%%%%%%%%%%%%%%%%%%%%%%%%%%%%%%%%%%%%%%%%%%%%%%%%%%%%%%%%%%%%

%%%%%%%%%%%%%%%%%%%%%%%%%%%%%%%%%%%%%%%%%%%%%%%%%%%%%%%%%%%%%%%%%%%%%%%%%%%%%%%%
%%%%%%%%%%%%%%%%%%%%%%%%%%%%%%%%%%%%%%%%%%%%%%%%%%%%%%%%%%%%%%%%%%%%%%%%%%%%%%%%
%%%%%%%%%%%%%%%%%%%%%%%%%%%%%%%%%%%%%%%%%%%%%%%%%%%%%%%%%%%%%%%%%%%%%%%%%%%%%%%%

%

\begin{thebibliography}{9}% Replace 9 by 99 if 10 or more references
%
%\bibitem{chinta} G. Chinta, `Analytic ranks of elliptic curves over cyclotomic fields', J.
%reine angew. Math. 544 (2002), 1324.
%\bibitem{adelmann} C. Adelmann, {\em The decomposition of primes in torsion point fields}, Lecture Notes in Mathematics 1761, Springer, 2001.

%\bibitem{arai} K. Arai, {\em On uniform lower bound of the Galois images associated to elliptic curves}, Tome 20, n. 1 (2008), p. 23-43.

%\bibitem{artin} E. Artin, {\em Algebraic Numbers and Algebraic Functions}, American Mathematical Society Chelsea Publishing, 2006. 

%\bibitem{bilu} Y. Bilu, P. Parent, {\em Serre's uniformity problem in the split Cartan case}, Annals of Mathematics, Volume 173 (2011), Issue 1, pp. 569-584.

%\bibitem{bilu2} Y. Bilu, P. Parent, M. Rebolledo, {\it Rational points on $X_0^+(p^r)$}, Ann. Inst. Fourier, to appear; arXiv:1104.4641v1

%\bibitem{lnm476} B. J. Birch, W. Kuyk (Editors), {\it Modular functions of one variable IV}, Lecture Notes in Mathematics 476, Berlin-Heidelberg-New York, Springer 1975.

%\bibitem{clark} P. L. Clark, B. Cook, J. Stankewicz, {\em Torsion points on elliptic curves with complex multiplication}, preprint.

%\bibitem{alina1} A. C. Cojocaru, {\em On the surjectivity of the Galois representations associated to non-CM elliptic curves (with
%an appendix by Ernst Kani)}, Canad. Math. Bull. 48 (2005), pp. 16-31.

%\bibitem{alina2} A. C. Cojocaru, C. Hall, {\em Uniform results for Serre's theorem for elliptic curves}, Int. Math. Res. Not. 2005,
%pp. 3065-3080.

%\bibitem{conrad} B. Conrad, K. Rubin (Editors), {\em Arithmetic Algebraic Geometry}, AMS, IAS/Park City Mathematics Series Volume 9, 2001

%\bibitem{stein} M. Derickx, S. Kamienny, W. Stein, M. Stoll, {\it Torsion points on elliptic curves over number fields of small degree}, in preparation (private communication).

%\bibitem{diamond} F. Diamond, J. Shurman, {\em A First Course in Modular Forms}, Graduate Texts in Mathematics 228, Springer-Verlag, 2nd Edition, New York, 2005.

%\bibitem{die} L. Dieulefait, {\em The level $1$ weight $2$ case of Serre's conjecture},  Rev. Mat. Iberoamericana, Volume 23, Number 3 (2007), pp. 1115-1124.

%\bibitem{edixhoven} B. Edixhoven, {\em Rational torsion points on elliptic curves over number fields (after Kamienny and Mazur)}. S\'eminaire Bourbaki, Vol. 1993/94. Ast\'erisque No. 227 (1995), Exp. No. 782, 4, 209-227.

%\bibitem{elkies1} N. Elkies, {\it Elliptic and modular curves over finite fields and related computational issues}, in Computational Perspectives on Number Theory: Proceedings of a Conference in Honor of A.O.L. Atkin (D.A. Buell and J.T. Teitelbaum, eds.; AMS/International Press, 1998), pp. 21-76.
%\bibitem{elkies2} N. Elkies, {\it Explicit Modular Towers}, in Proceedings of the Thirty-Fifth Annual Allerton Conference on Communication, Control and Computing (1997, T. Basar, A. Vardy, eds.), Univ. of Illinois at Urbana-Champaign 1998, pp. 23-32  (math.NT/0103107 on the arXiv).

%\bibitem{gross} B. H. Gross, {\it A tameness criterion for Galois representations associated to modular forms (mod $p$)}, Duke Mathematical Journal, Vol. 61, No. 2, pp. 445 - 517.

%\bibitem{flexor} M. Flexor, J. Oesterl\'e, {\it Sur les points de torsion des courbes elliptiques}, Ast\'erisque 183 (1990), 25?36.

%\bibitem{kleinfricke} R. Fricke, F. Klein, {\it Vorlesungen über die Theorie der elliptischen Modulfunctionen} (Volumes 1 and 2), B. G. Teubner, Leipzig 1890, 1892.

%\bibitem{fricke} R. Fricke, {\it Die elliptischen Funktionen und ihre Anwendungen}. Leipzig-Berlin: Teubner 1922.



%\bibitem{gaudron} \'E. Gaudron, G. R\'emond, {\it Th\'eor\`eme des p\'eriodes et degr\'es minimaux d'isog\'enies}, manuscript (2011), arXiv:1105.1230v1.

%\bibitem{hindry} M. Hindry, J. H. Silverman, {\it Sur le nombre de points de torsion rationnels sur une courbe elliptique}. C. R. Acad. Sci. Paris Ser. I Math., 329(2), (1999), pp. 97-100.

%\bibitem{ishii} N. Ishii, {\it Rational Expression for J-invariant Function in Terms of Generators of Modular Function Fields}, International Mathematical Forum, 2, 2007, no. 38, pp. 1877 - 1894.

%\bibitem{iwasawa} K. Iwasawa, {\it A note on Class Numbers of Algebraic Number Fields}, Abh. Math. Sem. Univ. Hamburg, 20 (1956), 257-258.

%\bibitem{iwasawa2} K. Iwasawa, {\it Local Class Field Theory}, Oxford University Press, USA (September 18, 1986).

%\bibitem{kamienny} S. Kamienny, B. Mazur, {\em Rational torsion of prime order in elliptic curves over number fields. With an appendix by A. Granville}. Columbia University Number Theory Seminar (New York, 1992). Astérisque No. 228 (1995), 3, 81-100.

%\bibitem{katz} N. Katz, {\it $p$-adic properties of modular forms}, Modular Functions of one Variable III, SLN 350, (1972), pp. 69-190.

%\bibitem{kenku2} M. A. Kenku, {\it The modular curve $X_0(39)$ and rational isogeny}, Math. Proc. Cambridge Philos. Soc. 85 (1979), pp. 21 - 23.

%\bibitem{kenku3} M. A. Kenku, {\it The modular curves $X_0(65)$ and $X_0(91)$ and rational isogeny}, Math. Proc. Cambridge Philos. Soc. 87 (1980), pp. 15 - 20.

%\bibitem{kenku4} M. A. Kenku, {\it The modular curve $X_0(169)$ and rational isogeny}, J. London Math. Soc. (2) 22 (1980), 239 - 244.

%\bibitem{kenku5} M. A. Kenku, {\it The modular curve $X_0(125)$, $X_1(25)$ and $X_1(49)$}, J. London Math. Soc. (2) 23 (1981), 415 - 427.

%\bibitem{kenku} M. A. Kenku, {\em On the number of $\Q$-isomorphism classes of elliptic curves in each $\Q$-isogeny class}, J. Number Th. 15 (1982), 199-202.

%\bibitem{khare0} C. Khare, {\em Serre's modularity conjecture: The level one case},  Duke Math. J. 134(3), (2006), pp. 557-589.

%\bibitem{khare} C. Khare, J-P. Wintenberger, {\em Serre's modularity conjecture. I},  Invent. Math. 178 (2009), no. 3, 485-504.

%\bibitem{khare2} C. Khare, J-P. Wintenberger, {\em Serre's modularity conjecture. II}, Invent. Math. 178 (2009), no. 3, 505-586.

%\bibitem{kida} M. Kida, {\em  Ramification in the Division Fields of an  Elliptic Curve}, Abh. Math. Sem. Univ. Hamburg 73 (2003), 195-207.

%\bibitem{kraus} A. Kraus, {\em Une remarque sur les points de torsion des courbes elliptiques}, C. R. Acad. Sci. Paris Sr. I Math. 321
%(1995), pp. 1143-1146.

%\bibitem{kubert} S. D. Kubert, {\it Universal bounds on the torsion of elliptic curves}, Proc. London Math. Soc. (3), 33 (1976), pp. 193 - 237.

%\bibitem{lang} S. Lang, {\it Fundamentals of Diophantine Geometry}, Springer; 1983 edition.

%\bibitem{vaintrob} E. Larson, D. Vaintrob, {\it Determinants of Subquotients of Galois Representations Associated to Abelian Varieties}, Journal of the Institute of Mathematics of Jussieu, recommended for publication. (also at arXiv:1110.0255)

%\bibitem{ligozat} G. Ligozat, {\it Courbes Modulaires de genre 1}, Bull. Soc. Math. France (1975), pp. 1 - 80.

%\bibitem{lozano} \'A. Lozano-Robledo, B. Lundell, {\it Bounds for the torsion of elliptic curves over extensions with bounded ramification}, International Journal of Number Theory, Volume: 6, Issue: 6 (2010), pp. 1293-1309.

%\bibitem{lozano2} \'A. Lozano-Robledo, {\it Formal groups of elliptic curves with potential good supersingular reduction}, to appear in the Pacific Journal of Mathematics.

%\bibitem{lozano1} \'A. Lozano-Robledo, {\it On the field of definition of $p$-torsion points on elliptic curves over the rationals}, preprint.

%\bibitem{lst} J. Lubin, J-P. Serre, J. Tate, {\it Elliptic curves and formal groups}, Lecture notes prepared in connection with the seminars held at the Summer Institute on Algebraic Geometry, Whitney Estate, Woods Hole, Massachusetts, July 6-31, 1964.

%\bibitem{maier} R. Maier, {\it On Rationally Parametrized Modular Equations}, J. Ramanujan Math. Soc. 24 (2009), pp. 1 - 73.

%\bibitem{masser} D. W. Masser, G. W\"ustholz, {\em Galois properties of division fields of elliptic curves}, Bull. London Math. Soc. 25
%(1993), pp. 247-254.

%\bibitem{mazurvelu} B. Mazur, J. V\'elu, {\it Courbes de Weil de conducteur 26}, C. R. Acad. Sc. Paris, t. 275 (1972), s\'erie A, pp. 743-745.

%\bibitem{mazur} B. Mazur, {\em Rational points on modular curves} (in \cite{serre3}), Proceedings of Conference on Modular Functions held in Bonn, Lecture Notes in Math. 601, Springer-Verlag, Berlin-Heiderberg-New York (1977), pp. 107-148.

%\bibitem{merel} L. Merel, {\em Bornes pour la torsion des courbes elliptiques sur les corps de nombres}, Invent. Math. 124 (1996), no. 1-3, 437-449.

%\bibitem{momose} F. Momose, {\em Rational points on the modular curves $X_{\text{split}}(p)$}, Compositio Math., 52 (1984),115-137.

%\bibitem{neron} A. N\'eron, {\em Mod\`eles minimaux des vari\'et\'es ab\'eliennes sur les corps locaux et globaux}, Inst. Hautes Etudes Sci. Publ. Math. (1964), No. 21, pp. 128.

%\bibitem{ogg} A. Ogg, {\it Rational points on certain elliptic modular curves}, Proc. Symp. Pure Math. XXIX, AMS, (1973) pp. 221 - 231. 

%\bibitem{parent0} P. Parent. {\it Bornes effectives pour la torsion des courbes elliptiques sur les corps de nombres}, J. Reine Angew. Math. 506 (1999).

%\bibitem{parent} P. Parent, {\it No $17$-torsion on elliptic curves over cubic number fields}, Journal de Th\'eorie des Nombres de Bordeaux 15 (2003), 831-838.

%\bibitem{pellarin} F. Pellarin, {\em Sur une majoration explicite pour un degr\'e d'isog\'enie liant deux courbes elliptiques}, Acta Arith.
%100 (2001), pp. 203-243.

%\bibitem{prasad} D. Prasad, C. S. Yogananda, {\em Bounding the torsion in CM elliptic curves}. C. R. Math. Acad. Sci. Soc. R. Can. 23 (2001), 1-5.

%\bibitem{rebolledo} M. Rebolledo, {\em Module supersingulier et points rationnels des courbes modulaires}, Th\`ese, Universi\'e Pierre et Marie Curie, 2004.

%\bibitem{ribet} K. A. Ribet, W. A. Stein, {\em Lectures on Serre's Conjectures}, in \cite{conrad}.

%\bibitem{serretate} J.-P. Serre, J. Tate, {\em Good reduction of abelian varieties}, Ann. of Math. 88 (1968), pp. 492 - 517.

%\bibitem{serre1} J-P. Serre, {\em Propri\'et\'es galoisiennes des points d'ordre fini des courbes elliptiques}, Invent. Math. 15 (1972), pp. 259-331.


%\bibitem{serre3} J.-P. Serre, D. B. Zagier (Editors), {\em Modular Functions of One Variable V: Proceedings International Conference}, University of Bonn, Sonderforschungsbereich Theoretische Mathematik, July 2-14, 1976: No. V (Lecture Notes in Mathematics 601).

%\bibitem{serre} J-P. Serre, {\it Points rationnels des courbes modulaires $X_0(N)$}, Seminaire Bourbaki, 1977/1978, No. 511.

%\bibitem{serre2} J-P. Serre, {\em Quelques applications du th\'eor\`eme de densit\'e de Chebotarev}, Publ. Math. IHES 54 (1981), pp. 123-201.

%\bibitem{serre5} J-P. Serre, {\em Sur les repr\'esentations modulaires de degr\'e 2 de $\Gal(\overline{\Q}/\Q)$}, Duke Math. J., 54(1), (1987), pp. 179-230.

%\bibitem{serre4} J-P. Serre, {\em Abelian $l$-adic representations and elliptic curves}, A K Peters/CRC Press; 3 edition (November 15, 1997).



%\bibitem{serretopics} J.-P. Serre, {\em Topics in Galois theory}, Second Edition, Research Notes in Mathematics, A K Peters/CRC Press (November 2, 2007).

%\bibitem{silverberg} A. Silverberg, {\em Torsion points on abelian varieties of CM-type}. Compositio Math. 68 (1988), no. 3, 241-249.

%\bibitem{silverman} J. H. Silverman, {\em The Arithmetic of Elliptic Curves}, Springer-Verlag, 2nd Edition, New York, 2009.

%\bibitem{silverman2} J. H. Silverman, {\em Advanced Topics in the Arithmetic of Elliptic Curves}, Springer-Verlag, New York.

%\bibitem{shimura}, G. Shimura, {\em Introduction to the Arithmetic Theory of Automorphic Functions}, Princeton University Press; 1st edition (August 1, 1971).

%\bibitem{wash} L. C. Washington, {\em Introduction to Cyclotomic Fields}, Springer-Verlag, New York.

\bibitem{bruinnajman} P. Bruin, F. Najman, {\em A criterion to rule out torsion groups for elliptic curves over number fields}, preprint.

\bibitem{cmtorsion} P. Clark, P. Corn, A. Rice, J. Stankewicz, {\em Computation on elliptic curves with complex multiplication};  arXiv:1307.6174v2 [math.NT]

\bibitem{fujita} Y. Fujita, {\em Torsion subgroups of elliptic curves in elementary abelian 2-extensions of $\Q$}, J. Number Theory 114 (2005), 124-134.

\bibitem{gonztorn} E. Gonz\'{a}lez-Jim\'{e}nez, J. Tornero, {\em Torsion of rational elliptic curves over quadratic fields},  Rev. R. Acad. Cienc. Exactas F\'{i}s. Nat. Ser. A Math. RACSAM 108 (2014), 923-934.

\bibitem{gonzlozano} E. Gonz\'{a}lez-Jim\'{e}nez, \'{A}. Lozano-Robledo, {\em Elliptic curves with abelian division fields}, preprint, available at http://alozano.clas.uconn.edu/research-articles.

\bibitem{jeonkim} D. Jeon, C. H. Kim, {\em On the arithmetic of certain modular curves}, Acta Arith. 112 (2004), 75-86.

\bibitem{jeonkimlee} D. Jeon, C. H. Kim, Y. Lee, {\em Families of elliptic curves over quartic number fields with prescribed torsion subgroups}, Math. Comp. 80 (2011), pp. 2395-2410.

\bibitem{jeonkimpark} D. Jeon, C. H. Kim, E. Park, {\em On the torsion of elliptic curves over quartic number fields}, J. London Math. Soc. (2) 74 (2006), pp. 1-12.

\bibitem{jeonkimschweizer} D. Jeon, C.H. Kim, A. Schweizer, {\em On the torsion of elliptic curves over cubic number fields}, Acta Arith. 113 (2004), 291-301.

\bibitem{kamienny} S. Kamienny, {\em Torsion points on elliptic curves and q-coefficients of modular forms}, Invent. Math. 109 (1992), 221-229.

\bibitem{kenku} M.A. Kenku {\em On the number of $\Q$-isomorphism classes of elliptic curves in each $\Q$-isogeny class}, J. Number Theory 15 (1982), 199-202.

\bibitem{kenkumomose} M.A. Kenku, F. Momose, {\em Torsion points on elliptic curves defined over quadratic fields}, Nagoya Math. J. 109 (1988), 125-149.

\bibitem{knapp} A.W. Knapp, {\em Elliptic Curves}, Princeton University Press, Princeton, NJ, 1992.

\bibitem{laskalorenz} M. Laska, M. Lorenz, {\em Rational points on elliptic curve over $\mathbb{Q}$ in elementary abelian 2-extensions of $\mathbb{Q}$}, J. Reine Angew Math. 355 (1985) 163-172.

\bibitem{lozanorobledo1} \'A. Lozano-Robledo, {\em On the field of definition of $p$-torsion points on elliptic curves over the rationals}, Mathematische Annalen, Vol 357, Issue 1 (2013), 279-305.

%\bibitem{lozanorobledo2} \'A. Lozano-Robledo, {\em Uniform boundedness in terms of ramification}, preprint.

\bibitem{mazur1} B. Mazur, {\it Rational isogenies of prime degree}, Inventiones Math. 44 (1978), pp. 129 - 162.

\bibitem{najman} F. Najman, {\em Torsion of rational elliptic curves over cubic fields and sporadic points on $X_1(n)$}, Math. Res. Letters, to appear

\bibitem{2adicimage} J. Rouse, D. Zureick-Brown, {\em Elliptic Curves Over $\mathbb{Q}$ and 2-adic Images of Galois}; arXiv:1402.5997v2 [math.NT]

%\bibitem{silverman} J. H. Silverman, {\em The Arithmetic of Elliptic Curves}, Springer-Verlag, 2nd Edition, New York, 2009.

\end{thebibliography}
\end{document}